\documentclass[12pt,reqno ]{amsart}

\usepackage{amsrefs}
\usepackage{amsthm}
\usepackage{amssymb}
\usepackage{amsfonts}
\usepackage{amsmath}
\usepackage{mathrsfs}
\usepackage{comment}
\usepackage{hyperref}
\numberwithin{equation}{section}

\newtheorem{theorem}{Theorem}[section]
\newtheorem{lemma}[theorem]{Lemma}
\newtheorem{proposition}[theorem]{Proposition}

\newtheorem{corollary}[theorem]{Corollary}

\theoremstyle{definition}

\newcommand{\abs}[1]{|#1|}
\newcommand{\C}{\ensuremath{\mathbb{C}^n}}

\newcommand{\Ew}{\ensuremath{\mathbb{E}_\omega}}
\newcommand{\R}{\ensuremath{\mathbb{R}}}
\newcommand{\Z}{\ensuremath{\mathbb{Z}}}
\newcommand{\Rd}{\ensuremath{\mathbb{R}^d}}
\newcommand{\D}{\ensuremath{\mathscr{D}}}
\newcommand{\Dw}{\ensuremath{\mathscr{D}_\omega}}
\newcommand{\J}{\ensuremath{\mathscr{J}}}
\newcommand{\F}{\ensuremath{\mathscr{F}}}

\renewcommand{\H}{\ensuremath{\mathcal{H}}}

\newcommand{\BMOW}{\ensuremath{{\text{BMO}}_W ^p}}

\newcommand{\BMOWpq}{\ensuremath{{\text{BMO}}_W ^{p, q}}}
\newcommand{\BMOWq}{\ensuremath{{\text{BMO}}_{W^{1 - p'}}^{p'} }}

\newcommand{\Mn}{\ensuremath{\mathcal{M}_{n }}(\mathbb{C})}

\newcommand{\Span}{\text{span}}
\newcommand{\V}[1]{\ensuremath{\vec{#1}}}
\newcommand{\inrd}{\ensuremath{\int_{\Rd}}}
\newcommand{\tr}{\ensuremath{\text{tr}}}
\newcommand{\ip}[2]{\ensuremath{\left\langle#1,#2\right\rangle}}
\newcommand{\W}[1]{\ensuremath{\widetilde{#1}}}
\newcommand{\T}[1]{\ensuremath{\text{#1}}}
\newcommand{\MC}[1]{\ensuremath{\mathcal{#1}}}
\renewcommand{\S}{\ensuremath{\text{Sig}_d}}

\begin{document}

\title[A Matrix weighted $T1$ and John-Nirenberg theorem]{A matrix weighted $T1$ theorem for matrix kernelled CZOs and a matrix weighted John-Nirenberg theorem}


\author[Joshua Isralowitz]{Joshua Isralowitz}
\address[Joshua Isralowitz]{Department of Mathematics and Statistics \\
SUNY Albany \\
1400 Washington Ave. \\
 Albany, NY  \\
12222}
\email[Joshua Isralowitz]{jisralowitz@albany.edu}

\begin{abstract}
In this paper, we will prove a matrix weighted $T1$ theorem regarding the boundedness of certain matrix kernelled CZOs on matrix weighted $L^p(W)$ for matrix A${}_p$ weights $W$.  Using some of the ideas from the proof, we will also establish a natural matrix weighted John-Nirenberg result that extends to the matrix setting (in the case when one of the weights is the identity) a very recent characterization of both S. Bloom's BMO space and the two weight boundedness of commutators by I. Holmes, M. Lacey, and B. Wick.
\end{abstract}


\maketitle



\section{Introduction}

Weighted norm inequalities for Calder\'{o}n-Zygmund operators (or CZOs for short) acting on ordinary $L^p(\Rd)$ is a classical topic that goes back to the 1970's with the seminal works  \cite{CF, HMW}. On the other hand, it is well known that proving matrix weighted norm inequalities for CZOs is a very difficult task and thus such matrix weighted norm inequalities have only recently been investigated (see \cite{TV, V} for specific details of these difficulties).  In particular, if $W : \Rd \rightarrow \Mn$ is positive definite a.e., then define $L^p(W)$ to be the space of measurable $\vec{f} : \Rd \rightarrow \C$ with norm \begin{equation*} \|\vec{f}\|_{L^p(W)} ^p = \inrd |W^\frac{1}{p} (x) \vec{f}(x) |^p \, dx. \end{equation*} It was proved by F. Nazarov and S. Treil, M. Goldberg, and A. Volberg, respectively in \cite{NT, G, V}, that certain CZOs are bounded on $L^p(W)$ when $1 < p < \infty$ if  $W$ is a matrix A${}_p$ weight, which means that \begin{equation} \label{MatrixApDef} \sup_{\substack{I \subset \R^d \\ I \text{ is a cube}}} \frac{1}{|I|} \int_I \left( \frac{1}{|I|} \int_I \|W^{\frac{1}{p}} (x) W^{- \frac{1}{p}} (t) \|^{p'} \, dt \right)^\frac{p}{p'} \, dx  < \infty \end{equation} where $p'$ is the conjugate exponent of $p$ (note that an operator $T$ acting on scalar functions can be canonically extended to $\C$ valued functions via the action $T$ on its coordinate functions.)

It is known that CZOs with matrix valued kernels acting on $\C$ valued functions appear very naturally in various branches of mathematics (and as a particular example see \cite{IM1} for extensive applications of matrix kernelled CZOs to geometric function theory.) Despite this and despite the fact that the theory of matrix weights has numerous applications to Toeplitz operators, multivariate prediction theory, and even to the study of finitely generated shift invariant subspaces of unweighted $L^p(\Rd)$ (see \cite{NT, Mo, V}), virtually nothing (with the exception of the results in \cite{IKP}, which will be discussed momentarily) has been published regarding matrix weighted norm inequalities for matrix kernelled CZOs or similar operators.

The purpose of this paper is therefore to investigate the boundedness of matrix kernelled CZOs on $L^p(W)$ when $W$ is a matrix A${}_p$ weight.  We will need to introduce some more notation before we state our main result.  It is well known (see \cite{G} for example) that for any matrix weight $W,$ any $1 < p < \infty$, and any cube $I$, there exists (not necessarily unique) positive definite matricies  $V_I, V_I '$ such that  $|I|^{- \frac{1}{p}} \|1_I W^\frac{1}{p}  \vec{e}\|_{L^p} \approx |V_I \vec{e}|$ and $|I|^{- \frac{1}{p'}} \|1_I W^{-\frac{1}{p}} \vec{e}\|_{L^{p'}} \approx |V_I ' \vec{e}|$ for any $\vec{e} \in \C$, where $\|\cdot \|_{L^p}$ is the canonical $L^p(\R^d;\C)$ norm and the notation $A \approx B$ as usual means that two quantities $A$ and $B$ are bounded above and below by an unimportant constant multiple of each other.  It is not difficult to show that we may choose $V_I$ and $V_I'$ in such a way that $\|V_I V_I ' \| \geq 1$ for any cube $I$.  We will say that $W$ is a matrix A${}_p$ weight if the product $V_I V_I'$ has uniformly bounded matrix norm with respect to all cubes $I \subset \R^d$ (notice that this condition is easily seen to be equivalent to \eqref{MatrixApDef},  see \cite{R} for example). Also note that we may choose $V_I = (m_I W)^\frac{1}{2}$ and $V_I ' = (m_I (W^{-1}))^{\frac{1}{2}}$ when $p = 2$, where $m_I  W$ is the average of $W$ on $I$, so that the matrix A${}_2$ condition takes on a particularly simple form that is very similar to the scalar A${}_2$ condition.

The kind of matrix kernelled CZOs that we will investigate, roughly speaking, have matrix kernels $K : \Rd \times \Rd \backslash \Delta \rightarrow \Mn$ (where $\Delta$ as usual is the diagonal in $\Rd \times \Rd$)  that should be thought of as ``getting along well with the the matrix weighted $W$."  More precisely, let $\MC{S} = \Span\{1_J \vec{v} : J \text{ is a cube }, \V{v} \in \C\}$ and assume that $T $ defines a bilinear form on $\MC{S}$, which as usual will be denoted by $\ip{T\V{f}}{\V{g}}$ for $\V{f}, \V{g} \in \MC{S}$.    Further, if $\V{f}, \V{g} \in \MC{S}$ with disjoint support, then we assume that
\begin{equation*} \ip{T\V{f}}{\V{g}} = \inrd \inrd  K(x, y) \vec{f} (y) \vec{g} (x)\, dy \, dx  \end{equation*} and that a similar statement holds for $T^*$  (which is a bilinear form on $\MC{S}$ that is defined in the usual way) and $(K(y, x))^*$.    Second, for each cube $I \subset \Rd$, assume that the following ``standard kernel conditions" hold: \begin{equation*}
\begin{split}
 \sup_{\substack{I \subset \R^d \\  I \text{ is a cube}}} \|V_I K(x,y) V_I ^{-1}\| &\lesssim\frac{1}{\abs{x-y}^d}, \\
  \sup_{\substack{I \subset \R^d \\  I \text{ is a cube}}} \left(\|V_I (K(x,y)-K(x',y)) V_I ^{-1}\| + \|V_I ' (K^*(y,x)-K^*(y,x'))(V_I ') ^{-1}\| \right)
  &\lesssim  \frac{\abs{x-x'}^\alpha}{\abs{x-y}^{d + \alpha}}
\end{split}
\end{equation*}
for all $x,x',y\in\R^d$ with $\abs{x-y}>2\abs{x-x'}$ where $\alpha$ is independent of $I$. Finally, assume that $T$ satisfies the ``weak boundedness property" \begin{equation*} \sup_{(i, j) \in \{1, \ldots, n\}^2}\sup_{\substack{J  \subseteq I \\  I, J \text{ are cubes}}}   \frac{1}{|J|}  \left(|\ip{ V_I T V_I ^{-1} (1_J \V{e}_j)}{ 1_J \V{e}_i}| + |\ip{ V_I ' T^* (V_I') ^{-1} (1_J \V{e}_j)}{ 1_J \V{e}_i}| \right)< \infty  \end{equation*} where $1_J$ is the indicator function of the cube $J$ and $\{\vec{e}_k\}_{k = 1}^n$ is the standard orthonormal basis of $\C$.  We will call such an operator a $(W, p)$-CZO. In particular, notice that our ``weak boundedness property" implies that the usual weak boundedness property \begin{equation*} \sup_{\substack{J \subseteq \R^d \\  J \text{ is a cube}}} |\ip{T (1_J \V{u})}{1_J \V{v}}_{L^2}| \lesssim |J| |\V{u}| | \V{v}| \end{equation*} for any $u, v \in \C$ (and clearly a similar statement can be made regarding the standard unweighted ``size condition" and ``cancellation conditions" for $K$.)

Moreover, if $1 < p < \infty$ and $W$ is a matrix A${}_p$ weight, then let $\BMOW$ be the space from \cite{IKP} of locally integrable functions $B : \Rd \rightarrow \Mn$ where  \begin{equation*}     \left\{
     \begin{array}{lr}
     \displaystyle   \sup_{\substack{I \subset \R^d \\  I \text{ is a cube}}} \frac{1}{|I|} \int_I \|W^\frac{1}{p} (x) (B(x) - m_I B) V_I ^{-1}  \|^p \, dx < \infty  & :\text{ if } 2 \leq p  < \infty \\
     \displaystyle   \sup_{\substack{I \subset \R^d \\  I \text{ is a cube}}} \frac{1}{|I|} \int_I \|W^{-\frac{1}{p}} (x) (B^* (x) - m_I B^*) (V_I ')^{-1}  \|^{p'} \, dx < \infty & : \text{ if } 1 < p \leq 2
     \end{array}
   \right..
\end{equation*}

 Let us make a few brief comments about $\BMOW$ and the results in \cite{IKP}.  Note that while one can imagine various definitions of a matrix weighted BMO space that could replace the classical BMO space relative to the added noncommutativity,  the space $\BMOW$ (as was discussed in \cite{IKP}) is ``the" correct one in the sense that $[T, B]$ is bounded on $L^p(W)$ when $W$ is a matrix A${}_p$ weight and $T$ is any of the Riesz transforms $R_j$ for $j = 1, \dots, d$ if and only if $B \in \BMOW$. Furthermore, for any dyadic grid in $\mathbb{R}$ and any interval $I$ in this grid, let  \begin{equation*} h_I ^1 = |I|^{-\frac{1}{2}} 1_I (x), \,  \,  \,  \, \, \, \,  h_I ^0 (x) = |I|^{-\frac{1}{2}} (1_{I_\ell} (x) - 1_{I_r} (x))  \end{equation*} where $I_\ell$ and $I_r$ are the left and right halves of $I$, respectively.  Now given any dyadic grid $\D$ in $\mathbb{R}^d,$  any cube $I = I_1 \times \cdots \times I_d$, and any $\varepsilon \in \{0, 1\}^{d}$, let $h_I ^\varepsilon = \Pi_{i = 1}^d h_{I_i} ^{\varepsilon_i}$.  It is then easily seen that $\{h_I ^\varepsilon\}_{\{I \in \D, \  \varepsilon \in \S\}}$  where $\S = \{0, 1\}^d \backslash \{\vec{1}\}$ is an orthonormal basis for $L^2(\Rd)$.

  If again $\D$ is any dyadic lattice in $\Rd$ then let $\pi_B$ be the dyadic paraproduct defined by \begin{equation} \label{ParaprodDef} \pi_B \vec{f} = \sum_{\varepsilon \in \S} \sum_{I \in \D} B_I ^\varepsilon m_I (\vec{f}) h_I ^\varepsilon   \end{equation} where $\ m_I (\vec{f})$ is the vector average of $\vec{f}$ over $I$ and $B_I ^\varepsilon$ is the constant matrix whose entries are the Haar coefficients with respect to $h_I ^\varepsilon$.
 It was then proved in \cite{IKP} that $\pi_B$ is bounded on $L^p(W)$ when $W$ is a matrix A${}_p$ weight if and only if $B \in \BMOW$ (more precisely, if and only if $B$ is in the obvious dyadic version of $\BMOW$ relative to $\D$.)  Finally, we remark that it is easy to prove that in the scalar weighted setting, a scalar function $b$ is in $\BMOW$ if and only if $b \in \text{BMO}$.  In fact, it was shown in \cite{MW} that if $1 < p < \infty$ and $w \in \text{A}_\infty$ then $b \in \text{BMO}$ if and only if the first condition in the definition of $\BMOW$ is true, and if $1 < p < \infty$ and $w^{1 - p'}  \in \text{A}_\infty$ then $b \in \text{BMO}$ if and only if the second condition in the definition of $\BMOW$ is true.

 We can now state the main result of this paper, which is the following matrix weighted $T1$ theorem (and gives further evidence of the ``correctness" of the space $\BMOW$.)    \begin{theorem}\label{T1Thm}   \noindent If $W$ is a matrix A${}_p$ weight and $T$ is a $(W, p)$-CZO on $\Rd$, then $T$ extends to a bounded operator on $L^p(W)$  if and only if $T1 \in \BMOW$ and $T^* 1 \in \BMOWq$. \end{theorem} \noindent Note that here we define $T1$ (and similarly define $T^* 1$) via its action on $H^1$ atoms as the matrix \begin{equation*} \ip{T1}{a_I} = \lim_{R \rightarrow \infty}  \ip{T 1_{[-R, R]^d}}{a_I}   \end{equation*} where the matrix $\ip{T 1_{[-R, R]^d}}{a_I}$ is defined by  \begin{equation*} \left(\ip{T 1_{[-R, R]^d}}{a_I}\right)_{ij} = \ip{T 1_{[-R, R]^d} \V{e}_j}{a_I \V{e}_i}  \end{equation*}  and $a_I$ is an atom with vanishing mean on $I$ and supported on $I$.    Now if $I$ and $Q$ are any cubes with $I \subseteq Q$ and as usual,  $Q^* = 2\sqrt{n} Q$ and $c_I$ is the center of $I$, then by elementary arguments we have that \begin{equation} \ip{T1}{a_I} = \ip{T 1_{Q^*}}{a_I} + \int_{\Rd \backslash Q^*} \left(\int_I [K(x, y) - K(c_I, y)] a_I(x) \, dx \right)\, dy  \label{T1IntDef} \end{equation} which by the (usual) cancellation condition on $K$ exists.

 Since $T$ is not necessarily bounded on $L^2$, we need to carefully define what we mean by $T1 \in \BMOW$.  To do this we first need to mention the Triebel-Lizorkin bounds from \cite{NT, V} when $d = 1$ and \cite{I} when $d > 1$.  In particular,   if $\V{f}_I ^\varepsilon$ is the vector of Haar coefficients of the scalar entries of $\V{f}$ and $\V{f} \in L^p(W)$ then     \begin{equation}   \|\vec{f}\|_{L^p(W)} ^p  \approx \inrd \left(\sum_{\varepsilon \in \S} \sum_{I \in \D} \frac{| V_I \vec{f}_I ^\varepsilon |^2}{|I|} 1_I(x) \right)^\frac{p}{2} \, dx.  \label{LpEmbedding}  \end{equation}  We then say that $T1 \in \BMOW$ if  \begin{equation} \label{T1inBMOW} \sup_{\substack{I \subseteq \Rd \\ I \textit{ is a cube }}}  \, \frac{1}{|I|} \int_I \left(\sum_{\varepsilon \in \S} \sum_{Q \subseteq I}  \frac{\|V_Q (T1)_I ^\varepsilon V_I^{-1} \|^2}{|Q|} 1_Q (x) \right)^\frac{p}{2} \, dx < \end{equation} where we use the obvious notation $(T1)_I ^\varepsilon = \ip{T1}{h_I ^\varepsilon}$.  However, thanks to \eqref{LpEmbedding} and Proposition \ref{IsralProp} (which will be stated later in the introduction), \eqref{T1inBMOW} says that
 \begin{equation*} T1 = \sum_{\varepsilon \in \S} \sum_{I \in \D} (T1)_I ^\varepsilon  h_I ^\varepsilon \end{equation*} is well defined for every dyadic grid $\D$ and satisfies $T1 \in \BMOW$ (and that a similar statement can be said of $T^* 1$.)

 Let us make the following interesting remark regarding necessity in Theorem \ref{T1Thm}.  Notice that under the assumptions of Theorem \ref{T1Thm} we have that $TB$ (again via its action on atoms) makes sense if $B \in L^\infty$ is an $n \times n$ matrix function.  Furthermore, a careful check of the proof of necessity in Theorem \ref{T1Thm} reveals that $TB \in \BMOW$ if \begin{equation*} \sup_{\substack{I \subseteq \Rd \\ I \textit{ is a cube }}} \|V_I B V_I ^{-1}\|_{L^\infty} < \infty. \end{equation*}

Let us now make some comments regarding the $(W, p)$-CZO conditions. First, by the matrix A${}_p$ condition we have the crucial symmetry that $T$ is a $(W, p)$-CZO if and only if $T^*$ is a $(W^{1 - p'}, p')$-CZO.   Second, note our kernel conditions are in fact very natural.  In particular, consider the simplest matrix kernelled CZO on $\mathbb{R}$, namely $\H M_A$  where $\H$ is the Hilbert transform and $M_A \vec{f} = A \vec{f}$ for a constant matrix $A \in \Mn$, which obviously has matrix kernel $(x - y)^{-1} A$.  Then $T = \H M_A$ is bounded on $L^p(W)$ if and only if $(x - y)^{-1} A$ satisfies the above kernel estimates with $\alpha = 1$. Obviously since in this case $T1 = T^*1 = 0$, sufficiency will follow from Theorem \ref{T1Thm}.   To see necessity, note that as $\H$ is bounded on $L^p(W)$ (see \cite{NT, V}), we have that $\H M_A$ is bounded if and only if $M_A$ is bounded on $L^p(W)$, and clearly (again via appropriate testing functions) a necessary condition for $M_A$ to be bounded on $L^p(W)$ is that \begin{equation} \label{Acond} \sup_{\substack{I \subseteq \R \\ I \text{ is a interval }}}\|V_I A V_I ^{-1}\| < \infty \end{equation}  which clearly is true if and only if $(x - y)^{-1} A$ satisfies the above kernel estimates with $\alpha = 1$ (and a similar statement can be made regarding the weak boundedness property). Obviously since \begin{equation*} \sum_{j = 1}^d R_j ^2 = - \text{Id} \end{equation*} where $R_j$ is the $j^{\text{th}}$ Riesz transform, similar statements can be made of the matrix kernelled CZOs $T_j = R_j A$.

Also, despite its restricted appearance, the definition of a $(W, p)$-CZO is indeed checkable.  In particular, since \begin{equation*} \|A\| \approx \left(\sum_{j = 1}^n |A\V{e}_j|^p \right)^\frac{1}{p} \approx \left(\sum_{j = 1}^n |A\V{e}_j|^{p'} \right)^\frac{1}{p'} \end{equation*} for any $n \times n$ matrix $A$, we have when $W$ is a matrix A${}_p$ weight that \begin{equation*} \|V_I K(x, y) V_I ^{-1} \| \approx  \left( \frac{1}{|I|} \int_I \left[ \frac{1}{|I|} \int_I \|W^{\frac{1}{p}}(s) K(x, y)  W^{- \frac{1}{p}} (t) \|^{p'} \, dt \right]^\frac{p}{p'} \, ds\right)^\frac{1}{p}  \end{equation*} and obviously we can estimate the other ``standard kernel condition" similarly.  Moreover, if $T _J $ is the matrix defined by $(T _J)_{ij} =  \ip{  T  (1_J \V{e}_j)}{ 1_J \V{e}_i} $  then (thanks to elementary linear algebra) we can rewrite the weak boundedness property as \begin{equation*}\sup_{\substack{J  \subseteq I \\  I, J \text{ are cubes}}}   \frac{1}{|J|}  \left(\| V_I T _J  V_I ^{-1}\| + \| V_I ' (T^*) _ J  (V_I') ^{-1} \| \right) < \infty  \end{equation*}  which as before says that the weak boundedness property is readily  checkable.

Furthermore, it is very easy to see that if $T$ is a scalar kernelled CZO  and $A$ is a constant matrix, then the operator $T_A := T M_A $  satisfies the $(W, p)$-CZO conditions if \begin{equation*} \sup_{\substack {I \subseteq \Rd \\ I \text{ is a cube}}} ( \|V_I A V_I ^{-1}\|  + \|V_I ' A^* (V_I ')^{-1}\|) < \infty.   \end{equation*}  Thus, we have the following immediate corollary to Theorem \ref{T1Thm}, which is clearly of interest itself.

 \begin{corollary} \label{T1SpecialCor} Let  $W$ be a matrix A${}_p$ weight, let  $T$ be a scalar kernelled CZO,  and let $A$ be a constant $n \times n$ matrix.  If $W$ and $A$ satisfy the condition \begin{equation*} \sup_{\substack {I \subseteq \Rd \\ I \text{ is a cube}}} ( \|V_I A V_I ^{-1}\|  + \|V_I ' A^* (V_I ')^{-1}\|) < \infty   \end{equation*} then $T_A : L^p(W) \rightarrow L^p(W)$ (where $T_A$ is defined as above) iff $T_A 1 \in \BMOW$ and $(T^*)_A  1 \in \BMOWq$. \end{corollary}

We should comment, however, that it seems unclear whether our ``weak boundedness property" is genuinely a weak boundedness property that is satisfied whenever $T$ is bounded on $L^p(W)$.  It would therefore be quite interesting to know whether one can replace our ``weak boundedness property" with a more local condition that is satisfied whenever $T$ is bounded on $L^p(W)$

Interestingly, the above considerations give a simple counterexample to Theorem \ref{T1Thm} when $T$ is \textit{not} a $(W, p)$-CZO.  Namely if $A$ is the $2 \times 2$ matrix given by $A_{ij} = 1 - \delta_{ij}$, then in fact $M_A  L^2(W) \nsubseteq L^2(W)$ when $W = \text{diag}(|x|^\beta, \ |x|^{-\beta})$ for $\frac{1}{2} < \beta < 1$ since \begin{equation*} \|W^\frac12 A W^{-\frac12} (1, 0)^T 1_{[0, 1]} \|_{L^2} ^2 = \int_0^1 x^{-2\beta} \, dx = + \infty \end{equation*} so that $T =\H M_A$ fails (rather miserably) to even map $L^2(W)$ into itself.


Let us now comment on the organization of the paper.  In the next section we will prove sufficiency in Theorem \ref{T1Thm}.  The proof strategy  will be to employ the by now standard technique of ``surgery" from non-homogenous analysis \footnotemark \footnotetext{Contrary to what is stated in the literature, it should be noted that that this technique was first used in \cite{NT} to prove matrix weighted inequalities for certain scalar CZOs} in a way that allows us to modify and combine the arguments in \cite{H} and \cite{NT}.

In the third section we will prove necessity in Theorem \ref{T1Thm}.  Note that the presence of $W^\frac{1}{p}$ and $W^{-\frac{1}{p}}$ in the definition of $\BMOW$ makes this a harder task than it normally would be in the unweighted setting.  To mitigate this, we will need an admittedly strange looking kind of ``matrix weighted John-Nirenberg" theorem (Lemma \ref{JNlemma}) which roughly says that a matrix function $B$ will be in $\BMOW$ if it satisfies a similar weighted BMO like condition with the weight replaced by the reducing operators $V_I$.  The rest of the proof of necessity will then follow from a modification of the classical techniques used to prove that $T : L^\infty \rightarrow \text{BMO}$ for any CZO $T$.

In the last section, we will use some ideas from the proof of necessity to prove the following John-Nirenberg type result that complements the classical weighted John-Nirenberg theorem from \cite{MW}.

\begin{theorem}  \label{MWMatrixThm}  Let $W$ be a matrix A${}_p$ weight.  If there exists $q > 1$ such that \begin{equation}  \label{VectorBMO1} \sup_{\substack{J \subset \R^d \\  J \text{ is a cube}}}  \frac{1}{|J|} \int_J |{V_J} ^{-1}  (\V{f} (x) - m_J \V{f})|^q \, dx < \infty \end{equation} then \begin{equation}  \label{VectorBMO2} \sup_{\substack{J \subset \R^d \\  J \text{ is a cube}}}  \frac{1}{|J|} \int_J |W^{-\frac{1}{p}}(x)   (\V{f} (x) - m_J \V{f})|^{p'} \, dx < \infty. \end{equation} Conversely, if (\ref{VectorBMO2}) holds then so does (\ref{VectorBMO1}) for $q = p$. \end{theorem}

Note that Theorem \ref{MWMatrixThm} should in fact be thought of as a special case of a two weighted result that was proved by the author after this paper was written.  In particular, suppose that we use the notation $V_J(W)$ (and similar notation elsewhere) to indicate that we are taking the reducing operator with respect to the matrix weight $W$.  Then it was proved in \cite{I1} (using in fact arguments from this paper) that the following quantities are equivalent for $q > 1$ with $q - 1$ small enough if $W, U$ are matrix A${}_p$ weights and $B$ is a locally integral $n \times n$ matrix function:
\begin{itemize}{}{}
\item [(a)] $\displaystyle \sup_{\substack{J \subseteq \Rd \\ J \text{ is a cube}} } \frac{1}{|J|} \int_J \|V_J(W)^{-1}  (B(x) - m_J B) V_J(U) \|^{q} \, dx $
\item [(b)] $\displaystyle \sup_{\substack{J \subseteq \Rd \\ J \text{ is a cube}} }   \frac{1}{|J|} \int_J \|W^{-\frac{1}{p}} (x) (B (x) - m_J B) (V_J ' (U))^{-1}  \|^{p'} \, dx $
\item[(c)] $\displaystyle  \sup_{\substack{J \subseteq \Rd \\ J \text{ is a cube}} }   \frac{1}{|J|} \int_J \| V_J(W) ^{-1} (B(x) - m_J B) U^\frac{1}{p} (x) \|^p \, dx. $
\end{itemize}
\noindent Note that this result clearly implies  Theorem \ref{MWMatrixThm} by setting $U = \text{Id}_{n \times n}$ (and also implies Lemma \ref{JNlemma} by setting $U = W$) and further note that these equivalences are natural and vital when trying to characterize the two matrix weighted boundedness of paraproducts and commutators with Riesz transforms (see \cite{I1}).

Unfortunately it appears to be rather mysterious as to how to recover a genuine matrix weighted version of the classical weighted John-Nirenberg theorem from \cite{MW} for either vector or matrix functions.  An exception to this is in the ``matrix weighted/matrix function" $p = 2$ and $U = W^{-1}$ setting.  In particular, the equivalence between $(a)$ and $(b)$ reads (after using the matrix A${}_2$ condition) that the quantity\begin{equation} \label{MatrixBMO1} \sup_{\substack{J \subseteq \Rd \\ J \text{is a cube}} } \frac{1}{|J|} \int_J \|(m_J W)^{-\frac12}  (B(x) - m_J B) (m_J W)^{-\frac12} \|^{q} \, dx   \end{equation} is equivalent to the quantity \begin{equation} \label{MatrixBMO2} \sup_{\substack{J \subseteq \Rd \\ J \text{is a cube}} }   \frac{1}{|J|} \int_J \|W^{-\frac{1}{2}} (x) (B^* (x) - m_J B^*) (m_J W)^{-\frac12}  \|^{2} \, dx   \end{equation} which (modulo the $q > 1$ ) recovers the classical weighted John-Nirenberg theorem from \cite{MW}.

 On the other hand, note that the equivalences between $(a), (b),$ and $(c)$ are in fact not new in the scalar setting and were proved rather recently in \cite{HLW} (for $q = 1$) and was used in \cite{HLW} to characterize the two scalar weighted boundedness of paraproducts and commutators with general CZOs.  Further, note that the proof in \cite{HLW} largely relies on the scalar weighted John Nirenberg theorem from \cite{MW} in conjunction with the fact that if $w, u \in \text{A}_p$ are scalar weights and $\nu = w^\frac{1}{p} u^{-\frac{1}{p}}$ then $\nu \in \text{A}_2$ and \begin{equation*} (m_J w )^{-\frac{1}{p}} (m_J  u )^{\frac{1}{p' }} \approx (m_J  \nu)^{-1}. \end{equation*} Moreover,  $(a)$ clearly reduces to the ordinary BMO condition in the one scalar weighted case, and thus conditions $(b)$ and $(c)$ are not needed to prove one scalar weighted norm inequalities for paraproducts and commutators (which is most likely why conditions like $(b)$ and $(c)$ in the one scalar weighted case have not appeared in the literature before \cite{HLW}.)

 Now checking the proof carefully, it is easy to see that Theorem \ref{MWMatrixThm}  extends word for word to the case when a $\C$ valued function $\vec{f}$ is replaced by an $\Mn$ valued function $B$, and this leads to a very natural result when viewed from the point of view of commutators (at least for $q = p$).  In particular, in the last section we will show the following result (which follows from some simple ideas in \cite{IKP})  \begin{proposition}  \label{CommProp} Let $K : \mathbb{R}^d \backslash \{0\} \rightarrow \mathbb{C}$  be not identically zero, be homogenous of degree $-d$, have mean zero over the unit sphere $\partial \mathbb{B}_d$, and satisfy $K \in C^{\infty} (\partial \mathbb{B}_d)$ (so in particular $K$ could be any of the Riesz kernels).  If $T$ is the (convolution) CZO associated to $K$,  $B$ is a locally integral $\Mn$ valued function, and $W$ is a matrix A${}_p$ weight then $[T, B^*] : L^p(W) \rightarrow L^p$ boundedly (directly) implies that \textit{both}  \eqref{VectorBMO1} and \eqref{VectorBMO2} are true (for $q = p$, and $\vec{f}$ replaced by $B$).   \end{proposition}

 Interestingly, Theorem  \ref{MWMatrixThm} (for $q = p$) in the scalar case, which says that
\begin{equation*}  \sup_{\substack{J \subset \R^d \\  J \text{ is a cube}}}  \frac{1}{w(J)} \int_J |  f (x) - m_J {f}|^p \, dx < \infty \end{equation*} if and only if \begin{equation*}  \sup_{\substack{J \subset \R^d \\  J \text{ is a cube}}}  \frac{1}{|J|} \int_J |f (x) - m_J f|^{p'} w^{1 -p'} (x) \, dx < \infty \end{equation*} is a special case of the two weight boundedness characterization of commutators with Riesz transforms in \cite{HLW} (which first appeared in \cite{B} in the special case of the Hilbert transform) when one of the weights is the constant $1$.

In fact, well after this paper was written, the author proved general two matrix weighted results in the preprint \cite{I1} (when both are matrix A${}_p$ weights) similar to the results in \cite{HLW} that extend and unify Theorem \ref{MWMatrixThm} and Proposition \ref{CommProp}.  While we will refer the interested reader to \cite{I1} for these results, we will only mention that the following equivalency to the conditions defining $\BMOW$ was proved in \cite{I1} \begin{proposition} \label{IsralProp} If $1 < p < \infty$ and $ W$  is a matrix A${}_p$ weights  then the following are equivalent:

\begin{itemize}{}{}

\item [(a')] $B \in \BMOW$ \\
\item[(b')]$\displaystyle \sup_{\substack{J \subset \R^d \\  J \text{ is a cube}}} \frac{1}{|J|} \int_J \|W^\frac{1}{p} (x) (B(x) - m_J B) (V_J) ^{-1}  \|^p \, dx < \infty $
\item[(c')] $\displaystyle \sup_{\substack{J \subset \R^d \\  J \text{ is a cube}}}  \frac{1}{|J|} \int_J \|W^{-\frac{1}{p}} (x) (B^* (x) - m_J B^*) (V_J ' )^{-1}  \|^{p'} \, dx < \infty.   $ \\

\end{itemize}
\end{proposition}

Note that we have in fact proven that (modulo Theorem \ref{T1Thm}) if $A$ is a constant matrix, then we have $\vec{f} \mapsto A \vec{f}$ is bounded on $L^p(W)$ when $W$ is a matrix A${}_p$ weight if and only if \eqref{Acond} (defined with respect to cubes) is true. As was mentioned earlier, \eqref{Acond} is a necessary condition easily obtained via testing functions.  On the other hand, it is easy to prove sufficiency without recourse to the Riesz transforms.  In particular, clearly (for $W$ not necessarily A${}_p$ and $A$ not necessarily constant)  $\vec{f} \mapsto A \vec{f}$ is bounded on $L^p(W)$ if and only if $W^{\frac{1}{p}} A W^{-\frac{1}{p}} \in L^\infty$.   However, if $u \in \Rd$ and $\{I_{k} ^{u}\}$ is a nested sequence of cubes whose intersection is $\{u\}$ and $1 < p \leq 2$ then the matrix A${}_p$ condition and the Lebesgue differentiation theorem gives us that \begin{align*} \|W^\frac{1}{p} (u) A W^{-\frac{1}{p}} (u) \|^{p} & = \lim_{k \rightarrow \infty} \frac{1}{|I_{k} ^{u}|} \int_{I_{k} ^{u}} \left(\frac{1}{|I_{k} ^{u}|} \int_{I_{k} ^{u}} \|W^\frac{1}{p} (x) A W^{-\frac{1}{p}} (y) \|^{p} \, dy \right) \, dx \\ & \leq \limsup_{k \rightarrow \infty} \frac{1}{|I_{k} ^{u}|} \int_{I_{k} ^{u}} \left(\frac{1}{|I_{k} ^{u}|} \int_{I_{k} ^{u}} \|W^\frac{1}{p} (x) A W^{-\frac{1}{p}} (y) \|^{p'} \, dy \right)^\frac{p}{p'} \, dx \\ & \lesssim \sup_{\substack{ I \subseteq \Rd \\ I \text{ is a cube}}}  \|V_I A V_I ' \| ^p \\ & \approx \sup_{\substack{ I \subseteq \Rd \\ I \text{ is a cube}}}  \|V_I A V_I ^{-1} \| ^p. \end{align*}  On the other hand if $p > 2$ then one can repeat these arguments by estimating  $\|W^{-\frac{1}{p}} (u) A^* W^{\frac{1}{p}} (u) \|^{p'}$ in a similar manner.

We end this introduction with two comments and an application of Theorem \ref{MWMatrixThm}.   First, while applications of Theorem \ref{T1Thm} most likely apply to cases when $T$ is apriori bounded on $L^2$ (that is, when $T$ is a ``CZO" in the traditional sense), there is no great difficulty in proving Theorem \ref{T1Thm} in our level of generality.  Second,  while this will not be needed in the rest of the paper, it is very easy to show (and rather curious) that $B, B^* \in \BMOW$ implies that $B \in \text{BMO}$, which will be proved at the end of the third section.

Finally, notice that Theorem \ref{MWMatrixThm} has an intriguing application to the Hilbert transform and matrix weighted BMO that complements the scalar results in \cite{MW}.  Namely, let $W$ be a matrix A${}_p$ weight that satisfies the B${}_{2, p}$ condition  defined by  \begin{equation*} \sup_{\substack{I \subseteq \mathbb{R} \\ I \text{ is an interval }}} |I| \int_{\mathbb{R} \backslash I} \frac{\|V_I ^{-1} W^\frac{1}{p} (t)\|}{|t - c_I|^2} \, dt < \infty. \end{equation*} Also, as in \cite{MW} let $\W{H}$ be the slightly modified Hilbert transform defined by \begin{equation*} \W{H} \V{f} (x) = \lim_{\epsilon \rightarrow 0^+} \int_{|x - y| > \epsilon} \left[\frac{1}{x-y} + \frac{1_{[-1, 1]^c} (y)}{y} \right] \V{f} (y) \, dy \end{equation*} (so that $\W{H} \V{f}$ is well defined when $\vec{f}$ is locally integrable.)

\begin{proposition} \label{MWHilbertProp} If $W$ is a matrix A${}_p$ weight satisfying the B${}_{2, p}$ condition, then for any $\vec{f}$ where $W^{-\frac{1}{p}} \vec{f} \in L^\infty$ we have \begin{equation*}  \sup_{\substack{I \subseteq \mathbb{R} \\ I \text{ is an interval }}} \frac{1}{|I|} \int_I |W^{-\frac{1}{p}}(x)   (\W{H} \V{f} (x) - m_I (\W{H} \V{f}))|^{p'} \, dx \lesssim \|W^{-\frac{1}{p}} \vec{f}\|_{L^\infty}.   \end{equation*} \end{proposition}

\noindent To prove Proposition \ref{MWHilbertProp} one only has to slightly modify the proof argument for necessity in Theorem \ref{T1Thm} to see that the first condition in Theorem \ref{MWMatrixThm} is satisfied.

 Note that the B${}_{2, p}$ condition is very natural and in fact satisfied automatically if $W$ is a matrix A${}_p$ weight for $1 < p \leq 2$.  To see this, let $w_I(t) = \|V_I ^{-1} W^\frac{1}{p} (t) \|$.  Clearly $w_I(t) ^p$ is a scalar A${}_p$ weight with A${}_p$ characteristic independent of $I$ (since  $|W^\frac{1}{p} \V{e} |^p$ is well known to be a scalar A${}_p$ weight with A${}_p$ characteristic comparable to $\|W\|_{\text{A}_p}$ and independent of $\V{e}$.)  Thus, we clearly have that $w_I$ is a scalar A${}_p$ weight, and as remarked in \cite{MW}, it is known then that \begin{equation*} |I| \int_{\mathbb{R} \backslash I} \frac{w_I(t)}{|t - c_I|^2} \, dt \lesssim m_I w_I \end{equation*}  and trivially by H\"{o}lder's inequality we have that \begin{equation*} \sup_{\substack{I \subseteq \mathbb{R} \\ I \text{ is an interval }}} m_I w_I \leq \sup_{\substack{I \subseteq \mathbb{R} \\ I \text{ is an interval }}} \left(\frac{1}{|I|} \int_I \|V_I ^{-1} W^\frac{1}{p} (t)\| ^p \, dt \right)^\frac{1}{p} < \infty. \end{equation*}

\section{Sufficiency}

     We now follow the notation of \cite{H} closely. From now on let $\D_0$ be the standard dyadic lattice. By assumption, for any $\omega \in \prod_{\Z} \{0, 1\}^d$ we have that $\pi ^{\omega} _{T1}$ and $(\pi ^{\omega} _{T^*1})^*$ are both bounded on $L^p(W)$ (with operator norms independent of $\omega$)  where the paraproducts are with respect to the dyadic lattice $\Dw := \D_0 \dot+ \omega$ (see \cite{H} for definitions).     Thus, if $\widetilde{T} ^\omega = T - \pi ^{\omega} _{T1} - (\pi ^{\omega} _{T^*1})^*$ then Theorem \ref{T1Thm} will follow if we can prove that \begin{equation*} \left|\Ew \ip{\widetilde{T}^\omega \vec{f}}{\vec{g}}_{L^2} \right| \lesssim \|\vec{f}\|_{L^p(W)} \|\vec{g}\|_{L^{p'}(W^{1 -p'})} \end{equation*} for all $\vec{f} \in L^2(\Rd, \C) \cap L^p(W)$ and $ \vec{g} \in L^2(\Rd, \C) \cap L^{p'} (W^{1-p'})$ (where the expectation on $\prod_{\Z} \{0, 1\}^d$ is with respect to the standard product measure.)  \\

While other choices of ``good" and ``bad" cubes can probably be made, we will follow the definition from \cite{NTV}.  In particular, we will say a cube $I \in \Dw$ is bad if there exists $J \in \Dw$ such that $\ell(J) \geq 2^r \ell(I)$ and \begin{equation*} d (I, \partial J) \leq \ell(I)^\gamma \ell(J)^{1 - \gamma} \end{equation*} where $\gamma  = \frac{\alpha}{2\alpha + 2d}$ and we will say that $I \in \Dw$ is good if it is not bad.  As is shown in \cite{H}, we can fix $r > 0$ such that \begin{equation*} \pi_{\T{bad}} = \mathbb{P}(\{\omega : I \dot+ \omega \T{ is bad }\}) < 1 \end{equation*} (in fact, the above probability is  independent of $I \in \D_0$.

 Now let $\D$ be any dyadic lattice, let $W$ be a matrix A${}_p$ weight, and let $\W{T}  = T - \pi_{T1} - \pi_{T^*1} ^*$ where the paraproducts are with respect to $\D$. Also let $\W{T} _{I, J}$  for fixed $\varepsilon, \varepsilon' \in \S$ be the matrix defined by \begin{equation*} (\W{T} _{I, J})_{ij} = \ip{\widetilde{T}  (h_I ^\varepsilon \V{e}_j)}{h_J ^{\varepsilon'}  \V{e}_i}  \end{equation*} (where for notational convenience we do not omit the subscript $``L^2$ when denoting the bilinear form $T$ defines on $\MC{S}$)  and define $T_{I, J}$ similarly.   Furthermore let $\W{T}^{I_0} := V_{I_0} \W{T} V_{I_0} ^{-1} $ and we will also let $\W{T} _{I, J} ^{I_0}$ for a fixed cube $I_0$ be defined by $\W{T} _{I, J} ^{I_0} =  V_{I_0} \W{T} _{I, J} V_{I_0} ^{-1}$ (also define ${T}^{I_0}$ and ${T} _{I, J} ^{I_0}$ in a similar manner).  Note that the specific value of $\varepsilon, \varepsilon' \in \S$ will not play any role in what follows and thus for notational ease will be suppressed in the above definitions.)

\begin{lemma} \label{TheHuntLem}     Fix $I_0 \in \D$,  fix $\varepsilon, \varepsilon' \in \S$, let $W$ be a matrix A${}_p$ weight and let $T$ be a $(W, p)-$CZO.  If $\ell(I) \leq \ell(J)$ and both $I$ and $J$ are contained in $I_0$ then\begin{equation}  \|   \W{T}^{I_0} _{I, J} \| \lesssim \frac{\ell(I)^\frac{d + \alpha}{2} \ell(J)^\frac{d + \alpha}{2}}{D(I, J)^{d + \alpha}}  \label{HaarMatrixEst}\end{equation} if $I$ is good, where $D(I, J) = |I| + |J| + d(I, J)$ and $d(I, J) = \T{dist}(I, J)$.   \end{lemma}

\begin{proof}
 Throughout the proof (and the rest of the paper) we will define $Tf$ for a ``nice" scalar function $f$ to be the matrix \begin{equation*} (Tf)_{ij} = \langle T(f \V{e}_j), \V{e}_i \rangle_{\C} \end{equation*} and we similarly define $\W{T} f$ and $\W{T}^{I_0} f$.   Note that many of the estimates needed for this proof are by now well known, so we will omit some details.  As in \cite{H}, we decompose $\{(I, J) \in \D \times \D : \ell(I) \leq \ell(J)\}$ as

\begin{align*} \{(I, J) : \ell(I) \leq \ell(J) \} & = \{(I, J) : d(I, J) > \ell(I) ^\gamma \ell(J)^{1 - \gamma} \} \cup \{(I, J) : I \subsetneq J \}  \\ & \cup \{(I, J) : I = J \} \cup \{(I, J) : d(I, J) \leq \ell(I) ^\gamma \ell(J)^{1 - \gamma} \T{ and } I \cap J = \emptyset \} \\ & := \Omega_{\T{out}} \cup \Omega_{\T{in}} \cup   \Omega_{\T{equal}} \cup \Omega_{\T{near}} \end{align*} and we will estimate $\W{T}^{I_0} _{I, J}$ for $(I, J)$ in each of these sets in a manner that is similar to the arguments in \cite{H}. Also by a simple and straightforward computation we have \begin{equation*} \langle \W{T} ^{I_0} h_I ^\varepsilon ,   h_I ^{\varepsilon'} \rangle = \W{T}_{I, J} ^{I_0} \end{equation*} and a similar result holds for ${T}_{I, J} ^{I_0}$.

Now if $I \cap J = \emptyset$ then it is very easy to see that $\W{T}^{I_0} _{I, J}  = {T}^{I_0} _{I, J}$ where ${T}^{I_0} _{I, J}$ is defined in the obvious way.  Thus, if $c_I$ is the center of $I$, then for $(I, J) \in \Omega_{\T{out}}$, we have by the cancellation of $h_I$ and the standard estimates\begin{align*}  \|\ip{\W{T}^{I_0} h_I ^{\varepsilon}}{h_J ^{\varepsilon'}} _{L^2}\|   & \leq \int_{\Rd} \, \int_{\Rd} |h_I ^{\varepsilon}(y) | | h_J ^{\varepsilon'}(x)| \|V_{I_0} (K(x, y)   -  K(x, c_I)) V_{I_0} ^{-1} \| \, dy \, dx \\ & \lesssim \frac{\ell(I)^\alpha}{d(I, J) ^{d + \alpha}}  \ell(I)^{\frac{d}{2}} \ell(J)^{\frac{d}{2}}. \end{align*}

If $d(I, J) > \ell(J)$ then $d(I, J) \approx D(I, J)$ so that \begin{equation*} \frac{\ell(I)^\alpha}{d(I, J) ^{d + \alpha}}  \ell(I)^{\frac{d}{2}} \ell(J)^{\frac{d}{2}} \lesssim \frac{\ell(I)^\frac{d + \alpha}{2} \ell(J)^\frac{d + \alpha}{2}}{D(I, J)^{d + \alpha}}  \end{equation*} since $\ell(I) \leq \ell(J)$.  On the other hand if $d(I, J) \leq \ell(J)$ then $D(I, J) \approx \ell(J)$ and since  $d(I, J) > \ell(I) ^\gamma \ell(J)^{1 - \gamma}$ we get \begin{equation*} \frac{\ell(I)^\alpha}{d(I, J) ^{d + \alpha}}  \ell(I)^{\frac{d}{2}} \ell(J)^{\frac{d}{2}} \lesssim \frac{\ell(I)^\alpha}{\{\ell(I) ^\gamma \ell(J)^{1 - \gamma}\} ^{d + \alpha}}  \ell(I)^{\frac{d}{2}} \ell(J)^{\frac{d}{2}} \approx \frac{\ell(I)^\frac{d + \alpha}{2} \ell(J)^\frac{d + \alpha}{2}}{D(I, J)^{d + \alpha}} \end{equation*} where here we have used the fact that $\gamma(d + \alpha) = \frac{\alpha}{2}$.

If $I = J$ then clearly again we have $\W{T}^{I_0} _{I,I} = {T}^{I_0} _{I, I}$.  Thus, we can easily estimate $\|\W{T}^{I_0} _{I, I}\| $ by utilizing the weak boundedness property in conjunction with the size estimate property of each $V_{I_0} K(x, y) V_{I_0}^{-1}$ (see \cite{H}, p. 13 for example).

As for $(I, J) \in \Omega_{\T{near}}$, the ``goodness" of $I$ tells us that $\ell(I) \leq \ell(J) < 2^r \ell(I)$.   Thus, we can estimate $\W{T}^{I_0} _{I, J}$ in exactly the same manner as we estimate  $\W{T}^{I_0} _{I, J}$ when $I = J$.

To finish the proof, we estimate $\W{T}^{I_0} _{I, J}$ when $(I, J) \in \Omega_{\T{in}}$. By a straight forward computation using the definition of $T$ on $L^\infty$ we have that  ``$\W{T} ^{I_0} 1 = 0$" in the sense that \begin{equation} \ip{\W{T}^{I_0} h_J}{1_Q} = \ip{\W{T}^{I_0} h_J}{1_{Q^c}} =  \ip{{T}^{I_0} h_J}{1_{Q^c}}_{L^2} \label{T1=0} \end{equation} whenever $J, Q \in \D$ with $J \subseteq Q$.    Pick $J_I$ satisfying $\ell(J_I) = \frac12 \ell(J)$  and $I \subseteq J_I \subseteq J$.  Thus,  we have \begin{align*} \ip{\W{T}^{I_0} h_I}{h_J}_{L^2} & = \ip{  \W{T}^{I_0} h_I}{  1_{J_I ^c} h_J} + \ip{\W{T}^{I_0} h_I}{ 1_{J_I} h_J} \\ & = \ip{{T}^{I_0} h_I}{1_{J_I ^c } h_J}_{L^2} + h_J(J_I) \ip{{\W{T}}^{I_0} h_I}{ 1_{J_I} } \\ & = \ip{{T}^{I_0} h_I  }{1_{J_I ^c} (h_J - h_J(J_I))}_{L^2}  \\ & \lesssim |J|^{-\frac12} \int_{J_I ^c} |{T}^{I_0} h_I(x)| \, dx  \end{align*}

Now if $\ell(I) \leq \ell(J) \leq 2^r \ell(I)$ then by the size and cancellation estimates we have  \begin{equation*}  |J|^{-\frac12} \int_{J_I ^c} |{T}^{I_0} h_I(x)| \, dx \lesssim \left(\frac{|I|}{|J|} \right)^\frac12 \approx \frac{\ell(I)^\frac{d + \alpha}{2} \ell(J)^\frac{d + \alpha}{2}}{D(I, J)^{d + \alpha}} \end{equation*}

However, if $\ell(J) > 2^r \ell(I)$ then ``goodness" gives us that $d(I, J_I ^c) \gtrsim \ell(I) ^\gamma \ell(J)^{1 - \gamma}$ so by the standard estimates \begin{align*}   |J|^{-\frac12} \int_{J_I ^c} |{T}^{I_0} h_I(x)| \, dx \  & \lesssim  \frac{\ell(I) ^{\alpha + \frac{d}{2}} }{\ell(J) ^\frac{d}{2}} \int_{\ell(I)^\gamma \ell(J) ^{1 - \gamma}} ^\infty  \frac{1}{r^{1+\alpha}} \, dr \\ &\approx  \frac{\ell(I) ^{\alpha + \frac{d}{2}}}{\ell(J)^\frac{d}{2} \left[\ell(I)^\gamma \ell(J)^{1 - \gamma} \right]^\alpha } \\ & \leq  \frac{\ell(I) ^{\alpha + \frac{d}{2}}}{\ell(J)^\frac{d}{2} \left[\ell(I)^\frac12 \ell(J)^{\frac12} \right]^\alpha } \\ & = \frac{\ell(I)^\frac{d + \alpha}{2} \ell(J)^\frac{d + \alpha}{2}}{D(I, J)^{d + \alpha}}\end{align*} (since clearly $\gamma \leq \frac12$) which completes the proof.

\end{proof}

We now prove the following ``surgical" lemma which is similar to Proposition $3.5$ in \cite{H} but exploits independence more. Note here that $\text{smaller}\{I, J\}$ is $I$ when $\ell(I) \leq \ell(J)$ and $J$ otherwise.
     \begin{lemma} \label{HytonenLem} Assume $T1 \in \BMOW$ and $T^*1 \in \BMOWq$. For $\vec{f} \in L^2 \cap L^p(W)$ and $\vec{g} \in L^2$ we have \begin{equation*}    \Ew \ip{\widetilde{T}^\omega \vec{f}}{\vec{g}}_{L^2}  = \frac{1}{\pi_{\text{good}}} \Ew \sum_{I, J \in \Dw} 1_\text{good} (\text{smaller}\{I, J\}) \ip{\widetilde{T} ^\omega _{I, J} \vec{f}_I}{\vec{g}_J}_{\C}\end{equation*} where  $\widetilde{T} ^\omega _{I, J}$ is defined as $\widetilde{T}  _{I, J}$ with respect to $\D_\omega$ (and where implicitly the first sum is also taken over all $\varepsilon, \varepsilon' \in \S$.)

      \end{lemma}

To prove this, however, we will first need the following result.

\begin{lemma}\label{T1Lem} If $W$ and $\W{T}^\omega$ are defined as above, then $\W{T}^\omega$ extends to a bounded operator on $L^2$ with operator norm independent of $\omega$.  \end{lemma}

\begin{proof} The proof is very similar to the dyadic proof of the classical ``T1" theorem proved in \cite{CT}, and we therefore only indicate where changes are needed. Suppose that $\V{f}$ and $\V{g}$ have finite Haar expansions with respect to $\D^\omega$ so that (after again suppressing the summations over $\varepsilon, \varepsilon \in \S$) \begin{equation*} \ip{\W{T}^\omega \V{f}}{\V{g}} = \sum_{I, J \in \D^\omega} \ip{\W{T}^\omega _{I, J}\V{f}_{I}}{\V{g}_J}_{\C}. \end{equation*}  By symmetry we can assume $\ell(I) \leq \ell(J)$, and as in \cite{CT} we outline the needed estimates for three cases (where $\gamma$ is as before): \begin{list}{}{}
\item $(1) \  d(I, \cup_i \partial P_i) \leq \ell(I)^\gamma \ell(J)^{1 - \gamma} $ where the $P_i$'s are the sons of $J$,
\item $(2) \ d(I, J) \geq \ell(I)^\gamma \ell(J)^{1 - \gamma}$,
\item $(3) \ d(I, \cup_i \partial P_i) \geq \ell(I) ^\gamma \ell(J)^{1 - \gamma} \text{ and } I \subsetneq J.$    \end{list}

Now if $B = T1$ and $\W{B} = T^*1$ then by definition \begin{equation} \label{TwDef} \pi_{B} ^\omega h_I +  \pi_{\W{B}}^\omega  h_I = \sum_{Q \subsetneq I} B_Q h_I (Q) h_Q - (\W{B} _I)^* \frac{1_I}{|I|} \end{equation} where $h_I(Q)$ is the constant value of $h_I$ on $Q \subsetneq I$.     We now look at $(I, J) \in (1)$. Since $\ell(I) \leq \ell(J)$ we have that $I \subseteq J$ or $I \cap J = \emptyset$.   Now by \eqref{TwDef} and the fact that $\pi_{B} ^\omega h_I \in L^p(W)$ we have  \begin{align*} |\W{T}^\omega _{I, J}| & \leq  |\ip{\W{T}^\omega h_I}{h_J 1_{(I^*)^c}}| + | \ip{\W{T}^\omega h_I}{h_J 1_{I^*}}| \\ & \leq \left\{
\begin{array}{ll}
      |\ip{{T} h_I}{h_J 1_{(I^*)^c}}_{L^2}| + |\ip{{T} h_I}{h_J 1_{I^*}}|  + |J|^{-\frac12} |\W{B}_I|   & \text{ if }I \subsetneq J \\
      |\ip{{T} h_I}{h_J 1_{(I^*)^c}}_{L^2}| + |\ip{{T} h_I}{h_J 1_{I^*}}| &  \text{ otherwise }
\end{array}
\right.
\\ & \lesssim \left(\frac{|I|}{|J|}\right)^\frac12
\end{align*} where the last inequality follows from the standard estimates involving the weak boundedness property, the size condition of $K$, and the cancellation condition of $K$.  Thus, arguing as in \cite{CT}, p. $8$ gives us that \begin{equation*} \sum_{I, J \in (1)} |\ip{\W{T}^\omega _{I, J}\V{f}_{I}}{\V{g}_J}_{\C} | \lesssim \|\V{f}\|_{L^2} \|\V{g}\|_{L^2}. \end{equation*}

However, if $(I, J) \in (2)$ then \eqref{TwDef} gives us that $\W{T}^\omega _{I, J} = T _{I, J}$, so arguing as in \cite{CT}, p. 11 (via using the cancellation condition on $K$) gives us that \begin{equation*} \sum_{I, J \in (2)} |\ip{\W{T}^\omega _{I, J}\V{f}_{I}}{\V{g}_J}_{\C} | \lesssim \|\V{f}\|_{L^2} \|\V{g}\|_{L^2}. \end{equation*}

Finally we handle the case $(I, J) \in (3)$. Now if $I \subset Q$ for $Q \in \D^\omega$  then the definition of $\W{B} = T^* 1$ allows one to easily check (since ``$\W{T}^\omega 1 = 0$" in the sense of \eqref{T1=0}) that \begin{equation*} \ip{\W{T}^\omega h_I}{1_Q}  = \ip{T h_I}{1_Q} - (\W{B}_I)^*   = -\ip{Th_I}{1_{Q^c}}_{L^2} \end{equation*} (which formally is trivial since $\ip{T h_I}{1_Q} - (\W{B}_I)^* = \ip{T h_I}{1_Q}  - \ip{Th_I}{1}$.)  Thus, if $I \subseteq J_I \subsetneq J$ where $J_I \in \D^\omega$ is a child of $J$, then \begin{align*} |\W{T}^\omega _{I, J}|   & \leq |\ip{\W{T}^\omega h_I}{h_J 1_{J \backslash J_I}}| + |\ip{\W{T}^\omega h_I}{h_J 1_{ J_I}}|  \\ & = |\ip{{T} h_I}{h_J 1_{J \backslash J_I}}_{L^2}| + |h_J(J_I) \ip{\W{T}^\omega h_I}{ 1_{  J_I}}_{L^2}| \\ & = |\ip{T h_I}{h_J 1_{I \backslash J_I}}_{L^2}| + |h_J(J_I) \ip{T h_I}{ 1_{ (J_I)^c}}_{L^2}|. \end{align*}  However, we can estimate both of these terms via the cancellation condition for $K$ to get that (see \cite{CT}, p. $15$) \begin{equation*} \sum_{I, J \in (3)} |\ip{\W{T}^\omega _{I, J}\V{f}_{I}}{\V{g}_J}_{\C} | \lesssim \|\V{f}\|_{L^2} \|\V{g}\|_{L^2}. \end{equation*}
\end{proof}


\noindent \textit{Proof of Lemma} \ref{HytonenLem}:
First, by Lemma \ref{T1Lem} we have that $\W{T}^\omega$ extends boundedly from $\Span \{h_I ^\varepsilon \V{v}: I \in \D_\omega, \varepsilon \in \S, \V{v} \in \C\}$ to $L^2$ with bounds independent of $\omega$.

For the rest of the proof we will fix $\varepsilon$ and $\varepsilon'$ and for the sake of notational ease write $h_I = h_I ^\varepsilon$ and $h_J = h_J ^{\varepsilon '}$.  One can then at the end sum up over all $\varepsilon, \varepsilon' \in \S$.  Now, as explained in \cite{H}, the badness of $I \dot+ \omega$ only depends on $\omega_j$ for $2^{-j} \geq \ell(I)$ whereas $I \dot+ \omega$ by definition itself depends on $2^{-j} < \ell(I)$.  Furthermore, \begin{equation*} \pi_{T1} ^\omega (h_{I \dot+ \omega} \V{e}_k) = \sum_{J \in \D (I)} (T1)_{J \dot+ \omega} m_{J \dot+ \omega} (h_{I \dot+ \omega} \V{e}_k) h_{J \dot+\omega} \end{equation*} and \begin{equation*}  (\pi_{T^*1} ^\omega)^* (h_{I \dot+ \omega} \V{e}_k)  = (T^*1)_{I \dot+ \omega} \frac{1_{I \dot+ \omega}}{|I \dot+ \omega|} \end{equation*} so that both $\pi_{T1} ^\omega (h_{I \dot+ \omega} \V{e}_k)$ and  $(\pi_{T^*1} ^\omega)^* (h_{I \dot+ \omega} \V{e}_k)$ only depend on $2^{-j} < \ell(I)$.

In other words, we have that $1_\text{good} (I \dot+ \omega)$ and $\ip{\widetilde{T} ^\omega (h_{I \dot+ \omega} \V{e}_k)}{ \vec{g}}_{L^2} \ip{\vec{f}}{ h_{I \dot+ \omega} \V{e}_k}_{L^2}$ are independent random variables.  Thus, by independence and the uniform $L^2$ boundedness of each $\W{T}^\omega$ on $L^2 \cap L^p(W)$ (which justifies both the Haar expansions and the interchange of expectations and summations)  we have

\begin{align*} \Ew & \ip{\widetilde{T}^\omega \vec{f}}{\vec{g}}_{L^2}  = \frac{1}{\pi_{\text{good}}} \sum_{k = 1}^n  \sum_{I \in \D_0 } \Ew  \left[1_\text{good} (I \dot+ \omega) \ip{\widetilde{T} ^\omega (h_{I \dot+ \omega} \V{e}_k)}{ \vec{g}}_{L^2} \ip{\vec{f}}{ h_{I \dot+ \omega} \V{e}_k}_{L^2}\right] \\ & = \frac{1}{\pi_{\text{good}}} \sum_{j, k = 1}^n  \sum_{I, J \in \D_0 }   \Ew  \left[1_\text{good} (I \dot+ \omega) \ip{\widetilde{T} ^\omega (h_{I \dot+ \omega} \V{e}_k )}{ h_{J \dot+ \omega} \V{e}_j}_{L^2} \overline{\ip{\vec{g}}{h_{J \dot+ \omega} \V{e}_j }_{L^2}}  \ip{\vec{f}}{ h_{I \dot +\omega} \V{e}_k}_{L^2}.\right] \end{align*}

Now if $\ell(I) > \ell(J)$ then again independence allows us to conclude that \begin{align*} & \frac{1}{\pi_{\T{good}}} \sum_{j, k = 1}^n   \sum_{\ell(I) > \ell(J) }   \Ew \left[1_\text{good} (I \dot+ \omega) \ip{\widetilde{T} ^\omega (h_{I \dot+ \omega} \V{e}_k )}{ h_{J \dot+ \omega} \V{e}_j}_{L^2} \overline{\ip{\vec{g}}{h_{J \dot+ \omega} \V{e}_j }_{L^2}}  \ip{\vec{f}}{ h_{I \dot +\omega} \V{e}_k}_{L^2}\right] \\ & = \frac{1}{\pi_{\T{good}}} \sum_{j, k = 1}^n  \sum_{\ell(I) > \ell(J) }   \Ew \left[1_\text{good} (I \dot+ \omega)\right] \Ew \left[\ip{\widetilde{T} ^\omega (h_{I \dot+ \omega} \V{e}_k )}{ h_{J \dot+ \omega} \V{e}_j}_{L^2} \overline{\ip{\vec{g}}{h_{J \dot+ \omega} \V{e}_j }_{L^2}}  \ip{\vec{f}}{ h_{I \dot +\omega} \V{e}_k}_{L^2}\right] \\ & = \sum_{j, k = 1}^n   \sum_{\ell(I) > \ell(J) }   \Ew \left[\ip{\widetilde{T} ^\omega (h_{I \dot+ \omega} \V{e}_k )}{ h_{J \dot+ \omega} \V{e}_j}_{L^2} \overline{\ip{\vec{g}}{h_{J \dot+ \omega} \V{e}_j }_{L^2}}  \ip{\vec{f}}{ h_{I \dot +\omega} \V{e}_k}_{L^2}\right] \end{align*} so that

\begin{align*}\Ew & \ip{\widetilde{T}^\omega \vec{f}}{\vec{g}}_{L^2} \\ & =  \frac{1}{\pi_{\text{good}}} \sum_{j, k = 1}^n  \sum_{\ell(I) \leq \ell(J) }   \Ew  \left[1_\text{good} (I \dot+ \omega) \ip{\widetilde{T} ^\omega (h_{I \dot+ \omega} \V{e}_k )}{ h_{J \dot+ \omega} \V{e}_j}_{L^2} \overline{\ip{\vec{g}}{h_{J \dot+ \omega} \V{e}_j }_{L^2}}  \ip{\vec{f}}{ h_{I \dot +\omega} \V{e}_k}_{L^2}\right] \\ & + \sum_{j, k = 1}^n   \sum_{\ell(I) > \ell(J) }   \Ew \left[\ip{\widetilde{T} ^\omega (h_{I \dot+ \omega} \V{e}_k )}{ h_{J \dot+ \omega} \V{e}_j}_{L^2} \overline{\ip{\vec{g}}{h_{J \dot+ \omega} \V{e}_j }_{L^2}}  \ip{\vec{f}}{ h_{I \dot +\omega} \V{e}_k}_{L^2}.\right] \end{align*}

However, arguing as before but not utilizing independence we have that \begin{align*} \Ew  \ip{\widetilde{T}^\omega \vec{f}}{\vec{g}}_{L^2} & =  \sum_{j, k = 1}^n  \sum_{\ell(I) \leq \ell(J) }   \Ew   \left[\ip{\widetilde{T} ^\omega (h_{I \dot+ \omega} \V{e}_k )}{ h_{J \dot+ \omega} \V{e}_j}_{L^2} \overline{\ip{\vec{g}}{h_{J \dot+ \omega} \V{e}_j }_{L^2}}  \ip{\vec{f}}{ h_{I \dot +\omega} \V{e}_k}_{L^2}\right]  \\ & +   \sum_{j, k = 1}^n  \sum_{\ell(I) >  \ell(J) }   \Ew   \left[\ip{\widetilde{T} ^\omega (h_{I \dot+ \omega} \V{e}_k )}{ h_{J \dot+ \omega} \V{e}_j}_{L^2} \overline{\ip{\vec{g}}{h_{J \dot+ \omega} \V{e}_j }_{L^2}}  \ip{\vec{f}}{ h_{I \dot +\omega} \V{e}_k}_{L^2}\right] \end{align*} so that \begin{align*} & \sum_{j, k = 1}^n  \sum_{\ell(I) \leq \ell(J) }   \Ew   \left[\ip{\widetilde{T} ^\omega (h_{I \dot+ \omega} \V{e}_k )}{ h_{J \dot+ \omega} \V{e}_j}_{L^2} \overline{\ip{\vec{g}}{h_{J \dot+ \omega} \V{e}_j }_{L^2}}  \ip{\vec{f}}{ h_{I \dot +\omega} \V{e}_k}_{L^2}\right]  \\ &  = \frac{1}{\pi_{\text{good}}} \sum_{j, k = 1}^n  \sum_{\ell(I) \leq \ell(J) }   \Ew  \left[1_\text{good} (I \dot+ \omega) \ip{\widetilde{T} ^\omega (h_{I \dot+ \omega} \V{e}_k )}{ h_{J \dot+ \omega} \V{e}_j}_{L^2} \overline{\ip{\vec{g}}{h_{J \dot+ \omega} \V{e}_j }_{L^2}}  \ip{\vec{f}}{ h_{I \dot +\omega} \V{e}_k}_{L^2}.\right] \end{align*}

Repeating these arguments almost word for word gives us that \begin{align*} & \sum_{j, k = 1}^n  \sum_{\ell(J) < \ell(I) }   \Ew   \left[\ip{\widetilde{T} ^\omega (h_{I \dot+ \omega} \V{e}_k )}{ h_{J \dot+ \omega} \V{e}_j}_{L^2} \overline{\ip{\vec{g}}{h_{J \dot+ \omega} \V{e}_j }_{L^2}}  \ip{\vec{f}}{ h_{I \dot +\omega} \V{e}_k}_{L^2}\right]  \\ &  = \frac{1}{\pi_{\text{good}}} \sum_{j, k = 1}^n  \sum_{\ell(J) < \ell(I) }   \Ew  \left[1_\text{good} (J \dot+ \omega) \ip{\widetilde{T} ^\omega (h_{I \dot+ \omega} \V{e}_k )}{ h_{J \dot+ \omega} \V{e}_j}_{L^2} \overline{\ip{\vec{g}}{h_{J \dot+ \omega} \V{e}_j }_{L^2}}  \ip{\vec{f}}{ h_{I \dot +\omega} \V{e}_k}_{L^2}. \right] \end{align*} so that

\begin{align*}  & \sum_{j, k = 1}^n  \sum_{I, J \in \D_0 }   \Ew   \left[\ip{\widetilde{T} ^\omega (h_{I \dot+ \omega} \V{e}_k )}{ h_{J \dot+ \omega} \V{e}_j}_{L^2} \overline{\ip{\vec{g}}{h_{J \dot+ \omega} \V{e}_j }_{L^2}}  \ip{\vec{f}}{ h_{I \dot +\omega} \V{e}_k}_{L^2}\right] \\ & =  \frac{1}{\pi_{\text{good}}} \sum_{j, k = 1}^n  \sum_{I, J \in \D_\omega }   \Ew  \left[ 1_\text{good} ( \T{smaller} \{I , J \}) \ip{\widetilde{T} ^\omega (h_{I } \V{e}_k )}{ h_{J } \V{e}_j}_{L^2} \overline{\ip{\vec{g}}{h_{J } \V{e}_j }_{L^2}}  \ip{\vec{f}}{ h_{I } \V{e}_k}_{L^2} \right] \end{align*}

\noindent which obviously then completes the proof. \hfill $\square$

As is discussed in \cite{NT}, an immediate consequence of \eqref{LpEmbedding}  and Khintchine's inequaliy is that for any $\V{f} \in L^p(W)$ we have \begin{equation} \label{LpWbound} \|\V{f}\|_{L^p(W)} \approx \left( \inrd \left(\sum_{\varepsilon \in \S} \sum_{I \in \D} \frac{|W^\frac{1}{p} (x) {\V{f}}_I ^\varepsilon |^2 }{|I|} 1_I (x)  \right)^\frac{p}{2} \, dx \right)^\frac{1}{p}. \end{equation}  As in \cite{NT}, note that \eqref{LpWbound} will be much more useful for us in this section than is \eqref{LpEmbedding}, though we will need \eqref{LpEmbedding} in the other two sections.

We can now prove Theorem $1.1$ \\

\noindent $\textit{Proof of Theorem} \, \ref{T1Thm} : $ By Lemma \ref{HytonenLem} we need to prove for $\vec{f} \in L^2 \cap L^p(W)$ and $\vec{g} \in L^2\cap L^{p'}(W^{1 - p'})$ that \begin{equation*} \sum_{I, J \in \D} 1_\text{good} (\text{smaller}\{I, J\} \left|\ip{\widetilde{T}  _{I, J} \vec{f}_I}{\vec{g}_J}_{\C}\right| \lesssim \|\V{f}\|_{L^p(W)} \|\V{g}\|_{L^{p'}(W^{1 - p'})} \end{equation*} independent of $\D$ (where $\W{T}$ is defined as in Lemma \ref{TheHuntLem}.)

To that end we first assume that $\ell(I) \leq \ell(J)$.  If \begin{equation*}  \W{T}^{(r)} _{I, J} = \left\{
     \begin{array}{lr}
     \displaystyle  \W{T}_{I, J} \text{ if } \ell(I) = 2^{-r} \ell(J) \\
     \displaystyle 0 \, \text{ otherwise} \end{array} \right.  \end{equation*}  then clearly it is enough to prove that
\begin{equation*} \sum_{\ell(I) \leq  \ell(J) }  1_\text{good} (I)\left|\ip{\widetilde{T} ^{(r)}  _{I, J} \vec{f}_I}{\vec{g}_J}_{\C}\right| \lesssim 2^{-\alpha r} \|\V{f}\|_{L^p(W)} \|\V{g}\|_{L^{p'}(W^{1 - p'})} \end{equation*} (where from now on we assume $I$ is good in appropriate sums and again we suppress the summation over all $\varepsilon, \varepsilon' \in \S$ when convenient).

Now let \begin{equation*} \V{f}_k = \sum_{\varepsilon \in \S} \sum_{I \in \D_k} \frac{\V{f}_I ^\varepsilon }{|I|^\frac12} 1_I   \end{equation*}  and similarly define $\V{g}_k$. It is then enough to show that \begin{align*}   \sum_{k \in \Z} \sum_{I \in \D_{k + r}, \,   J \in \D_{k } } & \int_{\Rd} \int_{\Rd}  \frac{1_I(s)}{|I|^{\frac12}} \frac{1_J(t)}{|J|^{\frac12}} \left|\ip{\W{T}_{I, J} ^{(r)}  \V{f}_{k + r}(s)}{\V{g}_{k} (t)}_{\C} \right| \, ds \, dt \\ & \lesssim  2^{-\frac{\alpha r}{2}} \|\V{f}\|_{L^p(W)} \|\V{g}\|_{L^{p'}(W^{1 - p'})} \end{align*}    \noindent   However, \eqref{LpWbound} exactly tells us that \begin{equation*} \|\vec{f} \|_{L^p(W)} \approx \|\{W^\frac{1}{p} \V{f}_k \}_{k \in \Z} \|_{L^p _{\ell^2}}  \end{equation*} Thus, it is enough to show that \begin{align*}   \sum_{k \in \Z} \sum_{I \in \D_{k + r}, \,   J \in \D_{k } } & \int_{\Rd} \int_{\Rd}  \frac{1_I(s)}{|I|^{\frac12}} \frac{1_J(t)}{|J|^{\frac12}} \left|\ip{\W{T}_{I, J} ^{(r)}  F_{k + r}(s)}{G_{k} (t)}_{\C} \right| \, ds \, dt \\ & \lesssim  2^{-\frac{r\alpha}{2}} \|W^{\frac{1}{p}} F \|_{L^p _{\ell^2}} \|W^{-\frac{1}{p}} G\|_{L^{p'}  _{\ell^2} } \end{align*}
for any $\C$ valued functions $F$ and $G$ defined on $\Rd \times \Z$ such that $W^\frac{1}{p} F  \in L^p _{\ell^2}$ and $W^{-\frac{1}{p'}} G  \in L^{p'}  _{\ell^2} = (L^p _{\ell^2})^*$.  However, if $S$ is the shift operator $(SF)_k (t) = F_{k - 1} (t)$ then clearly for all $m \in \Z$ we have $\|W^\frac{1}{p} S^m F\|_{L^p _{\ell^2}} = \|W^\frac{1}{p}  F\|_{L^p _{\ell^2}}$, which means that it is enough to show that \begin{align*}   \sum_{k \in \Z} \sum_{I \in \D_{k + r}, \,   J \in \D_{k } } & \int_{\Rd} \int_{\Rd}  \frac{1_I(s)}{|I|^{\frac12}} \frac{1_J(t)}{|J|^{\frac12}} \left|\ip{\W{T}_{I, J} ^{(r)}  F_{k }(s)}{G_{k} (t)}_{\C} \right| \, ds \, dt \nonumber \\ & \lesssim  2^{-\frac{r\alpha}{2}} \|W^{\frac{1}{p}} F \|_{L^p _{\ell^2}} \|W^{-\frac{1}{p}} G\|_{L^{p'}  _{\ell^2} }  \end{align*}

or equivalently \begin{align}   \sum_{k \in \Z} \sum_{I \in \D_{k + r}, \,   J \in \D_{k } } & \int_{\Rd} \int_{\Rd}  \frac{1_I(s)}{|I|^{\frac12}} \frac{1_J(t)}{|J|^{\frac12}} \left|\ip{ W^\frac{1}{p} (t) \W{T}_{I, J} ^{(r)}  W^{-\frac{1}{p}} (s) F_{k }(s)}{G_{k} (t)}_{\C} \right| \, ds \, dt \nonumber \\ & \lesssim  2^{-\frac{r\alpha}{2}} \| F \|_{L^p _{\ell^2}} \| G\|_{L^{p'}  _{\ell^2} }   \label{FinalKerEst1}  \end{align}

Now if $s \in I \in \D_{k + r}$ and $t \in J \in \D_k$ and $\rho_{I, J}$ is defined by \begin{equation*} \rho_{I, J}   = \frac{\ell(I)^\frac{ \alpha}{2} \ell(J)^\frac{ \alpha}{2}}{D(I, J)^{ d + \alpha}} \end{equation*} then \begin{equation*} \rho_{I, J} \lesssim \MC{K}_k(s, t)  \end{equation*} where \begin{equation*} \mathcal{K}_k(s, t) = \frac{ 2^{-\frac{r\alpha}{2}} 2^{-k\alpha} }{(2^{-k} + |s - t|)^{d + \alpha}}. \end{equation*}

We now proceed in a manner similar to that in \cite{PTV}, p. $13 - 14$.  By standard arguments for estimating Poisson-type kernels, we clearly have \begin{equation*} \frac{ 2^{-\frac{r\alpha}{2}} 2^{-k\alpha} }{(2^{-k} + |s - t|)^{d + \alpha}} \lesssim 2^{-\frac{r\alpha}{2}} \sum_{j = 0}^\infty 2^{-j\alpha } \frac{1_{B_{\infty} (s, 2^{j - k})}(t) }{|B_{\infty}(s, 2^{j-k})|} \end{equation*} where $B_{\infty} (s, R)$ denotes the ball (i.e. cube) of radius $R$ with center $s$ in the $\ell_\infty$ norm on $\Rd$.  However, it is not hard to show that \begin{equation*} \frac{1_{B_{\infty} (s, 2^{j - k})} (t)}{|B_{\infty}(s, 2^{j-k})|} \lesssim \int_{[0, 1)^{d}} \left[\sum_{\MC{C} \in {\D}_{k - j - 2 }, u} \frac{1_{\MC{C}}(s) 1_{\MC{C}}(t)}{|\MC{C}|} \right] \, du   \end{equation*}  where in general  \begin{equation*} {\D}_{m, u} = \left\{ 2^{-m} \left[[u, u + 1)^d + \ell\right] \right\}_{\ell \in \Z^d} \end{equation*} for $m \in \Z$.  In particular, if $\mathbb{P}$ is the uniform probability distribution on $[0, 1)^d$, then it is straight forward to see that \begin{equation*} \mathbb{P} (\{u : (s, t) \in (I+u) \times (I+u) \text{ for some } I \in \D_m \}) \geq \left(\frac{3}{4}\right)^d \ \end{equation*} as long as $|s - t|_{\infty} \leq 2^{-m - 2}$. Combining these estimates, we arrive at \begin{equation} 1_I(s) 1_J(t) \rho_{I, J} \lesssim 2^{-\frac{r\alpha}{2}} \sum_{j = 0}^\infty 2^{-j\alpha} \int_{[0, 1)^d}  \left[ \sum_{\MC{C} \in {\D}_{k - j - 2, u}} \frac{1_{\MC{C}}(s) 1_{\MC{C}}(t)}{|\MC{C}|} \right] \, du \label{FinalKerEst2}\end{equation}

Plugging $(\ref{FinalKerEst2})$ into $(\ref{FinalKerEst1})$  and letting $\MC{T}_{I, J} ^{(r)}$ be defined by \begin{equation*} \MC{T}_{I, J} ^{(r)} = \frac{\W{T} ^{(r)}_{I, J}}{\rho_{I, J} |I|^\frac{1}{2} |J|^\frac{1}{2}} \end{equation*}  tells us that we need to show (uniformly in $j$ and $u$) that

\begin{align*}      \sum_{k \in \Z} \sum_{\MC{C} \in {\D}_{k - j - 2, u}} & \int_{\Rd} \, \int_{\Rd}   \frac{1_{\MC{C}}(s) 1_{\MC{C}}(t)}{|\MC{C}|} \left|\ip{ W^\frac{1}{p} (t) {\MC{T}}_{I_{s} ^{k+r}, J_t ^k} ^{(r)}  W^{-\frac{1}{p}} (s) F_{k }(s)}{G_{k} (t)}_{\C} \right|  \, ds \, dt  \\ & \lesssim  \|F\|_{L^p _{\ell^2}} \|G\|_{L^{p'} _{\ell^2}} \end{align*} uniformly in $u$ and $j$, where $I_{s} ^{k+r}$ is the unique $I \in \D_{k + r}$ containing $s$ and $J_t ^k$ is defined similarly.  However, we can estimate

\begin{align*} & \left| \ip{W^{\frac{1}{p}} (t) (\MC{T}_{I_{s} ^{k+r}, J_t ^{k}} ^{(r)} )  W^{-\frac{1}{p}} (s) F_k (s)}{G_k(t)}_{\C} \right| \\ &  \leq  \left\|W^\frac{1}{p} (t) V_\MC{C} ^{-1}  \right\| \left\| W^{-\frac{1}{p}} (s)  V_\MC{C} \right\|  \\ & \times    \left\|  V_\MC{C} (\MC{T}_{I_{s} ^{k+r}, J_t ^{k}} ^{(r)} ) V_\MC{C} ^{-1}  \right\| |F_k(s)|     |G_k (t)| \\ & \lesssim  \left\|W^\frac{1}{p} (t) V_\MC{C} ^{-1}  \right\| \left\| W^{-\frac{1}{p}} (s)  V_\MC{C} \right\| |F_k(s)| |G_k (t)|  \end{align*} by the definition of $\MC{T}_{I, J} ^{(r)}$ and Lemma \ref{TheHuntLem} (since $I_{s} ^{k+r} \cup J_t ^{k} \subseteq \MC{C}$).

Thus, it is enough to show (uniformly) that \begin{align*}      \sum_{k \in \Z} \sum_{\MC{C} \in \D_{k - j - 2, u}} & \int_{\Rd} \, \int_{\Rd}   \frac{1_{\MC{C}}(s) 1_{\MC{C}}(t)}{|\MC{C}|} \left\|W^\frac{1}{p} (t) V_\MC{C} ^{-1}  \right\| \left\| W^{-\frac12} (s)  V_\MC{C} \right\| |F_k(s)| |G_k (t)|  \, ds \, dt  \\ & \lesssim  \|F\|_{L^p _{\ell^2}} \|G\|_{L^{p'} _{\ell^2}} \end{align*} or again utilizing ``shift operators," it is enough to show that \begin{align*}      \sum_{k \in \Z} \sum_{\MC{C} \in \D_{k, u}} & \int_{\Rd} \, \int_{\Rd}   \frac{1_{\MC{C}}(s) 1_{\MC{C}}(t)}{|\MC{C}|} \left\|W^\frac{1}{p} (t) V_\MC{C} ^{-1}  \right\| \left\| W^{-\frac12} (s)  V_\MC{C} \right\| |F_k(s)| |G_k (t)|  \, ds \, dt \\ & \lesssim  \|F\|_{L^p _{\ell^2}} \|G\|_{L^{p'} _{\ell^2}} \end{align*}

To finally finish the proof, we clearly need to show that $\MC{A}_{u, m}$ is bounded on $L^p_{\ell^2}$ uniformly in $u$ and $m$, where  \begin{equation*}(\MC{A}_{u, m} F)_k (t) = \sum_{\MC{C} \in \D_{m, u}  } \frac{1_\MC{C}(t)}{|\MC{C}|} \int_\MC{C} \left\|W^\frac{1}{p} (t) V_\MC{C} ^{-1}  \right\| \left\| W^{-\frac12} (s)  V_\MC{C} \right\| F_k(s) \, ds. \end{equation*}   However, this is exactly the content of the proof of Lemma $14.2$ in \cite{NT} (more precisely, the authors in fact prove that $\MC{A}_{u, m}$ is bounded on $L^p_{\ell^2}$ uniformly in $u$ and $m$ in order to obtain a slightly different result.)

Thus, for $\vec{f} \in L^2 \cap L^p(W)$ and $\vec{g} \in L^2\cap L^{p'}(W^{1 - p'})$ we have that \begin{equation*} \sum_{\ell(I) \leq \ell(J) } 1_\text{good} (I) \left|\ip{\widetilde{T}  _{I, J} \vec{f}_I}{\vec{g}_J}_{\C}\right| \lesssim \|\V{f}\|_{L^p(W)} \|\V{g}\|_{L^{p'}(W^{1 - p'})} \end{equation*} whenever $W$ is a matrix A${}_p$ weight and $T$ is a $(W, p)$-CZO.

However, the latter is true if and only if $W^{1-p'}$ is a matrix A${}_{p'}$ weight and $T^*$ is a $(W^{1-p'}, p')$-CZO.
Furthermore, we clearly have \begin{equation*} (\W{T}_{I, J}) ^* = {\left((\W{T} )^*\right)}_{J, I}  = (\W{T^*})_{J, I}  \end{equation*}  so that \begin{align*} \sum_{\ell(J)  < \ell(J) } 1_\text{good} (J) \left|\ip{\widetilde{T}  _{I, J} \vec{f}_I}{\vec{g}_J}_{\C}\right| & =  \sum_{\ell(J)  < \ell(J) } 1_\text{good} (J) \left|\ip{(\W{T^*})_{J, I} \vec{g}_J}{\vec{f}_I}_{\C}\right| \\ & \lesssim \|\V{g}\|_{L^{p'}(W^{1 - p'})}  \|\V{f}\|_{L^p(W)} \end{align*} which completes the proof.

\section{Necessity}
To prove necessity we will need the stopping time from \cite{I,IKP}.  In particular,  assume that $W$ is a matrix A${}_p$ weight. For any cube $I \in \D$ and some fixed $\lambda_1, \lambda_2  > 0$ (that will be specified later,) let $\J(I)$ be the collection of maximal $J \in \D(I)$ such that \begin{equation}     \|V_J V_I ^{-1} \|^p > \lambda_1 \  \  \text{   or   }  \ \  \| V_J ^{-1} V_I \|^{p'} > \lambda_2. \label{STDef} \end{equation}   Also, let $\F(I)$ be the collection of dyadic subcubes of $I$ not contained in any cube $J \in \J(I)$.  Clearly $J \in \F(J)$ for any $J \in \D(I)$ if $\lambda_1, \lambda_2 > 1$.

Let $\J^0 (I) := \{I\}$ and inductively define $\J^j(I)$ and $\F^j(I)$ for $j \geq 1$ by $\J^j (I) := \bigcup_{J \in \J^{j - 1} (I)} \J(J)$ and $\F^j (I) := \bigcup_{J \in \J^{j - 1} (I)} \F(J)$. Clearly the cubes in $\J^j(I)$ for $j > 0$ are pairwise disjoint.  Furthermore, since $J \in \F(J)$ for any $J \in \D(I)$, we have that $\D(I) = \bigcup_{j = 0}^\infty \F^j(I)$.

Note if $W$ is a matrix A${}_p$ weight then for $\lambda_1 $ large enough (independent of $W$) and $\lambda_2  \approx \|W\|_{\text{A}_p} ^{\frac{p'}{p}}$, we have that $|\bigcup \J ^j (I)| \leq 2^{-j} |I|$ for every $I \in \D$ (see \cite{I, IKP}). While we will not need it, it is interesting to notice that Lemma $3.1$ in \cite{V} easily implies that this conclusion is true when $W$ is a matrix A${}_{p, \infty}$ weight (see \cite{V} for the definition).

We will now use this stopping time to prove the following matrix weighted John-Nirenberg lemma which, as was mentioned in the introduction,  will be crucial for the proof of necessity.     \begin{lemma} \label{JNlemma} Let $1 < p , q < \infty$ and suppose that $W$ is a matrix A${}_p$ weight. If  $\BMOWpq$ is the space of matrix functions $B$ such that  \begin{equation*}  \sup_{\substack{I \subset \R^d \\  I \text{ is a cube}}} \frac{1}{|I|} \int_I \|V_I (B(x) - m_I B) V_I ^{-1} \|^q \, dx < \infty,   \end{equation*} then we have that \begin{equation*} \bigcup_{1 < q < \infty} \BMOWpq \subseteq \BMOW. \end{equation*}  \end{lemma}

\begin{proof} Let $B \in \BMOWpq$ for some $q > 1$ so by dyadic Littlewood-Paley theory,  \begin{equation} \label{BAssump}\sup_{\substack{I \subset \R^d \\  I \text{ is a cube}}} \frac{1}{|I|} \int_{I} \left( \sum_{\varepsilon \in \S} \sum_{J \in \D(I)} \frac{\|V_I B_J ^\varepsilon V_I ^{-1} \| ^2}{|J|} 1_J (x)  \right) ^\frac{q}{2} \, dx < \infty. \end{equation}  However, by Theorem $1.3$ in \cite{IKP} we have that $B \in \BMOW$ if and only if \begin{equation*} \sup_{\substack{I \subset \R^d \\  I \text{ is a cube}}} \, \frac{1}{|I|} \sum_{\varepsilon \in \S} \sum_{J \in \D(I)}  \|V_J B_J ^\varepsilon V_J ^{-1}\|^2 < \infty \end{equation*}  which by the classical John-Nirenberg theorem is equivalent to \begin{equation}  \label{Btoprove} \sup_{\substack{I \subset \R^d \\  I \text{ is a cube}}}  \frac{1}{|I|} \int_{I} \left( \sum_{\varepsilon \in \S}  \sum_{J \in \D(I)} \frac{\|V_J B_J ^\varepsilon V_J ^{-1} \| ^2}{|J|} 1_J (x)  \right) ^\frac{q}{2} \, dx < \infty. \end{equation}

Clearly by H\"{o}lder's inequality we can assume that $1 < q \leq 2$.  Note that  $J \in \F(K)$ implies that $\|V_J V_K^{-1}\| \lesssim 1$ and $\|V_K V_J^{-1}\| \lesssim \|W\|_{\text{A}_p} ^\frac{1}{p}$, so that for fixed $I \in \D$,

\begin{align}
&  \frac{1}{|I|} \int_{I} \left( \sum_{\varepsilon \in \S} \sum_{J \in \D(I)} \frac{\|V_J B_J ^\varepsilon V_J ^{-1} \| ^2}{|J|} 1_J (x)  \right) ^\frac{q}{2} \, dx \nonumber \\ & =
\frac{1}{|I|} \int_I \left( \sum_{j = 1} ^\infty \sum_{K \in \J^{j - 1} (I)} \sum_{\varepsilon \in \S} \sum_{J \in \F(K)}  \frac{\|V_J B_J ^\varepsilon V_J ^{-1} \|^2 }{|J|} 1_J (x) \right) ^\frac{q}{2} \, dx \nonumber \\ & \lesssim \frac{\|W\|_{\text{A}_p} ^\frac{q}{p}}{|I|}  \int_I \left(  \sum_{j = 1} ^\infty \sum_{K \in \J^{j - 1} (I)}  \sum_{\varepsilon \in \S} \sum_{J \in \D(K)} \frac{\|V_K B_J ^\varepsilon  V_K ^{-1} \|^2 }{|J|} 1_J (x) \right) ^ \frac{q}{2} \, dx \nonumber \\ & \leq  \frac{\|W\|_{\text{A}_p} ^\frac{q}{p}}{|I|}    \sum_{j = 1} ^\infty \sum_{K \in \J^{j - 1} (I)} \int_K \left( \sum_{\varepsilon \in \S} \sum_{J \in \D(K)} \frac{\|V_K B_J ^\varepsilon V_K ^{-1} \|^2 }{|J|} 1_J (x) \right) ^\frac{q}{2} \, dx. \label{EstTwo} \end{align}  However by \eqref{BAssump} we have that  \begin{align*} (\ref{EstTwo})  & \lesssim \frac{\|W\|_{\text{A}_p} ^\frac{q}{p}}{|I|}   \sum_{j = 1} ^\infty \sum_{K \in \J^{j - 1} (I)} |K|   \lesssim \|W\|_{\text{A}_p} ^\frac{q}{p'} \sum_{j = 1}^\infty 2^{-(j - 1)} < \infty \end{align*} which proves \eqref{Btoprove}. \end{proof}

We can now finally prove
 \begin{lemma} Let $W$ be a matrix A${}_p$ weight.  If $T$ is a $(W, p)$-CZO that is bounded on $L^p(W)$ then $T1 \in \BMOW$ and $T^* 1 \in \BMOWq$.
 \end{lemma}

Note that by duality it is enough to show that $T1 \in \BMOW$ when $W$ be a matrix A${}_p$ weight and $T$ is a $(W, p)$-CZO that is bounded on $L^p(W)$.

\begin{proof}   By the reverse H\"{o}lder's inequality we can pick (and fix) $\epsilon'$ such that \begin{equation*} \sup_{\substack{I \subset \R^d \\  I \text{ is a cube}}} \frac{1}{|I|} \int_I \|W^{-\frac{1}{p}} (x) V_I\|^{p' + \epsilon'} \, dx < \infty. \end{equation*}  Let $\epsilon > 0$, which will be determined momentarily and let $p_\epsilon = p + \epsilon$.   Combined with Lemma \ref{JNlemma}, the following lemma will complete the proof of necessity.

Now for any dyadic grid $\D$ and $I, Q \in \D$ with $Q \subseteq I$ we have by \eqref{T1IntDef} and the fact that $T$ is bounded on $L^p(W)$ that \begin{align*} \ip{T1}{h_Q ^\varepsilon} & = \ip{T 1_{I^*}}{h_Q ^\varepsilon}_{L^2}  + \int_Q  \left(\int_{\Rd \backslash I^*} [K(x, y) - K(c_I, y)]  \, dy \right)\, h_Q ^\varepsilon (x) dx   \\ & = \ip{F_{I, 1}}{h_Q ^\varepsilon}_{L^2} + \ip{F_{I, 2}}{h_Q ^\varepsilon}_{L^2} \end{align*} where \begin{equation*} F_{I, 1} (x)  = T 1_{I^*} (x), \ \ \ \ \ F_{I, 2} (x) = 1_I (x) \int_{\Rd \backslash I^*} [K(x, y) - K(c_I, y)]  \, dy \end{equation*}  Then by \eqref{LpEmbedding},  Lemma \ref{JNlemma}, and (unweighted) dyadic Littlewood Paley theory, our definition of $T1 \in L^p(W)$ tells that it is enough to prove that  \begin{equation*} \sup_{\substack{I \subset \R^d \\  I \text{ is a cube}}}\,  \frac{1}{|I|} \int_I \| V_I (F_{I, i} (x) )V_I ^{-1} \|^{p_\epsilon} \, dx < \infty \end{equation*} for $i = 0, 1$.  Let us denote these two supremums by $(1)$ and $(2)$.

We first estimate $(1)$ as follows. Fix some vector $\vec{e} \in \C$.  Then by (unweighted) duality with respect to the measure $dm_I = |I|^{-1} 1_I (x) \, dx$ and the linearity of $T$ \begin{align*}   \left(\frac{1}{|I|} \int_I  \abs{V_I T(1_{I^*}) (x) V_I ^{-1} \vec{e} } ^{p_\epsilon}  \, dx\right)^\frac{1}{p_\epsilon} & =  \sup_{\|\vec{g}\|_{L^{p_\epsilon ' }(dm_I)} = 1} \abs{\ip{  T(1_{I^*}V_I ^{-1} \vec{e} )}{ V_I  \vec{g}}_{L^2 (dm_I)}} \\ & = \sup_{\|\vec{g}\|_{L^{p_\epsilon '}(dm_I)} = 1} \frac{1}{|I|} \abs{\ip{  T(1_{I^*}V_I ^{-1} \vec{e} ) }{ V_I 1_I \vec{g}}_{L^2 }} \end{align*} where $p_\epsilon ' =  \frac{1 + \epsilon}{\epsilon}. $  However, assuming that $T$ is bounded on $L^p(W)$, we have that \begin{equation*} \frac{1}{|I|} \abs{ \ip{  T(1_{I^*}V_I ^{-1} \vec{e})  }{  V_I 1_I \vec{g} }_{L^2} } \leq  \frac{1}{|I|}\|1_{I^*} V_I ^{-1} \vec{e} \|_{L^p(W)} \|  1_I   V_I \vec{g} \|_{L^{p'} (W ^{1 - p'})}. \end{equation*} For the first term notice that a standard A${}_p$ weight argument with easy modifications in the matrix case shows that $|V_{I^*} \vec{e}| \lesssim |V_I \vec{e}|$ which means that the first term is bounded independent of $I$.

As for the second term, let $q_\epsilon = \frac{1 + \epsilon}{1 + \epsilon - \epsilon p'}$ so if $q_\epsilon '$ is the conjugate exponent of $q_\epsilon$ then we have that $p ' q_\epsilon ' = (1 + \epsilon)/\epsilon = p_\epsilon '$.  Thus, H\"{o}lder's inequality gives us that \begin{align*} |I| ^{-\frac{1}{p'}} \|  1_I   V_I \vec{g} \|_{L^{p'} (W ^{1 - p'})} & = \left( \frac{1}{|I|} \int_I |W^{-\frac{1}{p} } (x) V_I \vec{g} (x) |^{ p'} \, dx \right) ^\frac{1}{p'} \\ & \leq \left( \frac{1}{|I|} \int_I \|W^{-\frac{1}{p}} (x) V_I \| ^{q_\epsilon p'} \, dx \right)^\frac{1}{q_\epsilon p'} \left( \frac{1}{|I|} \int_I |\vec{g}(x)| ^{q_\epsilon ' p' } \, dx \right)^\frac{1}{q_\epsilon ' p ' } \\ & =  \left( \frac{1}{|I|} \int_I \|W^{-\frac{1}{p}} (x) V_I \| ^{q_\epsilon p'} \, dx \right)^\frac{1}{q_\epsilon p'}.\end{align*} Picking $\epsilon > 0$ small enough so that \begin{equation*} q_\epsilon p' = \frac{p'(1 + \epsilon)}{1 + \epsilon - \epsilon p'} \leq p' + \epsilon' \end{equation*} in conjunction with the reverse H\"{o}lder inequality gives us that
\begin{equation*} \sup_I |I| ^{-\frac{1}{p'}} \|  1_I   V_I \vec{g} \|_{L^{p'} (W ^{p' - 1})} < \infty \end{equation*}

 Finally estimating $(2)$ is easy, since by assumption $V_I K(x, y) V_I ^{-1}$ satisfies the standard CZ kernel estimates uniformly with respect to $I$.  Thus,  we have that $(2)$ is finite with constant independent of $I$ which completes the proof.

\end{proof}

We will end this section with the proof that $B, B^* \in \BMOW$ implies that $B \in \text{BMO}$.  In particular, using the fact that the trace and matrix norms are equivalent, we have \begin{align*} \|B\|_{\text{BMO}} ^2  & = \sup_{J \in \D} \, \frac{1}{|J|}  \sum_{\varepsilon \in \S} \sum_{I \in \D(J)} \|(B_I^\varepsilon) ^* B_I ^\varepsilon \|  \\ & \approx \sup_{J \in \D} \, \frac{1}{|J|}  \sum_{\varepsilon \in \S}  \sum_{I \in \D(J)}  \tr (B_I ^\varepsilon) ^* B_I ^\varepsilon  \\ & =
\sup_{J \in \D} \, \frac{1}{|J|}  \sum_{\varepsilon \in \S}  \sum_{I \in \D(J)}  \tr (V_I (B_I ^\varepsilon) ^* V_I ^{-1}) (V_I  B_I ^\varepsilon V_I^{-1})
\\ & \leq \sup_{J \in \D} \, \frac{1}{|J|}  \sum_{\varepsilon \in \S}  \sum_{I \in \D(J)}  \|V_I (B_I ^\varepsilon) ^* V_I ^{-1}\|  \,  \|V_I  B_I ^\varepsilon V_I^{-1}\| \\ & \leq \|B\|_{\W{\BMOW}} \|B^*\|_{\W{\BMOW}} \end{align*}
where \begin{equation*}\|B\|_{\W{\BMOW}} = \left( \sup_{J \in \D} \, \frac{1}{|J|}  \sum_{\varepsilon \in \S}  \sum_{I \in \D(J)}  \|V_I B_I ^\varepsilon V_I ^{-1}\| ^2 \right)^\frac12 \approx \|B\|_{\BMOW}  \end{equation*} when $W$ is a matrix A${}_p$ weight (see \cite{IKP}).

\section{Matrix weighted John-Nirenberg theorem}

We will now prove Theorem \ref{MWMatrixThm}.  To shorten the proof, we will first prove the following lemma     which is elementary yet interesting in its own right.

\begin{lemma}  \label{EmbeddingLem} Let $\{\V{\lambda}_I ^\varepsilon\}_{\{I \in \D, \varepsilon \in \S\}}$ be a Carleson sequence of vectors.  That is,  \begin{equation*}\|\lambda\|_* ^2 =  \sup_{J \in \D} \frac{1}{|J|} \sum_{\varepsilon \in \S} \sum_{I \in \D(J)} |\V{\lambda}_I ^\varepsilon |^2  < \infty. \end{equation*}  Then if $W$ is a matrix A${}_p$ weight and $B$ is any locally integrable $\Mn$ valued function, we have that \begin{equation*} \inrd \left(\sum_{\varepsilon \in \S} \sum_{I \in \D} \frac{|m_I ( B W^{- \frac{1}{p}}  )  V_I \V{\lambda}_I ^\varepsilon  |^2}{|I|} 1_I (x) \right)^\frac{p}{2} \, dx \lesssim \|B\|_{L^{p}} ^p\end{equation*} \end{lemma}

\begin{proof}The proof is similar to the proof of the ``matrix weighted Carleson embedding theorem" in \cite{IKP}.  We  will show that   \begin{align} \int_{\Rd} & \left( \sum_{\varepsilon \in \S} \sum_{I \in \D} \frac{ | m_I(B W^{-\frac{1}{p}}) V_I \vec{\lambda}_I ^\varepsilon | ^2 }{|I|} \chi_I(t) \right)^\frac{p}{2} \, dt
\nonumber \\ & \leq \int_{\Rd}  \left( \sum_{\varepsilon \in \S} \sum_{I \in \D} \frac{ (|\vec{\lambda}_I ^\varepsilon  | m_I \|V_I W^{-\frac{1}{p}} B^* \|)^2 }{|I|} \chi_I(t) \right)^\frac{p}{2} \, dt \label{CarEmbedParEst}\\ & \lesssim \|B\|_{L^p} ^p \nonumber  \end{align} for any $B \in L^p(\Rd;\C)$.

Now let \begin{equation*} \tilde{A} = \sum_{\varepsilon \in \S} \sum_{I \in \D} |\vec{\lambda}_I ^\varepsilon  | h_I ^\varepsilon \end{equation*} and let \begin{equation*} M_W ' B^* (x) = \sup_{\D \ni I \ni x}  m_I \|V_I W^{-\frac{1}{p}} B^*\| \end{equation*} Clearly for any $\D \ni I \ni x$ we have that \begin{equation*} m_I \|V_I W^{-\frac{1}{p}} B^*\| \leq m_I ( M_W ' B^*) \end{equation*} so that \begin{align*}  \eqref{CarEmbedParEst} & \leq \int_{\Rd}  \left( \sum_{\varepsilon \in \S} \sum_{I \in \D} \frac{ (|\vec{\lambda}_I ^\varepsilon | m_I( M_W ' B^*))^2 }{|I|} \chi_I(t) \right)^\frac{p}{2} \, dt
\\ & \lesssim \|\pi_{\tilde{A}} (M_W ' B^*)\|_{L^p} \\ &  \lesssim \|\lambda\|_* \|M_W ' B^*\|_{L^p}. \end{align*}

However, it is easy to see that \begin{equation*} \|M_W ' \|_{L^p \rightarrow L^p} \lesssim \|W\|_{\text{A}_p} ^{\frac{1}{p - 1}} \end{equation*} by using some simple ideas from \cite{G} (see \cite{IKP}).   \end{proof}

We can now prove Theorem \ref{MWMatrixThm}. \\

\noindent \textit{Proof of Theorem} \ref{MWMatrixThm}. By standard arguments it is enough to consider the supremums in Theorem \ref{MWMatrixThm} over some fixed dyadic lattice $\D$ and prove bounds independent of the choice of $\D$.   First assume that (\ref{VectorBMO1}) is true and note that by H\"{o}lder's inequality that we can obviously assume $1 < q \leq 2$.  Then by dyadic Littlewood-Paley theory, \begin{equation*} \sup_{J \in \D} \frac{1}{|J|} \int_J \left(\sum_{\varepsilon \in \S} \sum_{I \in \D(J)} \frac{| V_J ^{-1}  \vec{f}_I ^\varepsilon |^2}{|I|} 1_I(x) \right)^\frac{q}{2} \, dx < \infty   \end{equation*}

However, by \eqref{LpEmbedding} applied to $W^{1 - p'}$ and Theorem $3.1$ in \cites{NTV}, we need to prove that \begin{equation*} \sup_{J \in \D} \frac{1}{|J|} \int_J \left(\sum_{\varepsilon \in \S} \sum_{I \in \D(J)} \frac{| V_I '  \vec{f}_I ^\varepsilon|^2}{|I|} 1_I(x) \right)^\frac{q}{2} \, dx  < \infty \end{equation*}

To that end, by the matrix A${}_p$ property we have \begin{align*}  \sup_{J \in \D} \, \frac{1}{|J|} \int_J & \left(\sum_{\varepsilon \in \S} \sum_{I \in \D(J)} \frac{| V_I '  \vec{f}_I ^\varepsilon |^2}{|I|} 1_I(x) \right)^\frac{q}{2} \, dx \\ & \lesssim \sup_{J \in \D} \,  \frac{1}{|J|}  \int_J \left(\sum_{j = 1}^\infty \sum_{\varepsilon \in \S} \sum_{K \in \J^{j - 1}(J)} \sum_{I \in \F(K) } \frac{\| V_I'  V_K   \|   | V_K ^{-1} \vec{f}_I ^\varepsilon|^2}{|I|} 1_I(x) \right)^\frac{q}{2} \, dx \\ &  \lesssim  \sup_{J \in \D} \,  \frac{1}{|J|}  \int_J \left(\sum_{j = 1}^\infty \sum_{\varepsilon \in \S} \sum_{K \in \J^{j - 1}(J)} \sum_{I \in \D(K) } \frac{| V_K ^{-1} \vec{f}_I ^\varepsilon |^2}{|I|} 1_I(x) \right)^\frac{q}{2} \, dx \\ & \leq  \sup_{J \in \D} \,  \frac{1}{|J|}  \sum_{j = 1}^\infty \sum_{K \in \J^{j - 1}(J)} \frac{1}{|J|} \int_K \left( \sum_{\varepsilon \in \S} \sum_{I \in \D(K) } \frac{| V_K ^{-1} \vec{f}_I ^\varepsilon|^2}{|I|} 1_I(x) \right)^\frac{q}{2} \, dx \\ & \lesssim  \sup_{J \in \D} \,  \frac{1}{|J|}  \sum_{j = 1}^\infty \sum_{K \in \J^{j - 1}(J)} |K|  < \infty \end{align*} where in the second to last inequality we use \eqref{VectorBMO1} in conjunction with dyadic Littlewood-Paley theory.

For the converse, \eqref{LpEmbedding} applied to $W^{1 - p'}$ and the matrix A${}_p$ condition gives us that  \eqref{VectorBMO2} is equivalent to \begin{align*} \sup_{J \in \D} \frac{1}{|J|} \, \int_J & \left( \sum_{\varepsilon \in \S} \sum_{I \in \D(J)} \frac{|V_I ^{-1} \V{f}_I ^\varepsilon |^2}{|I|} \, 1_I (x)  \right)^\frac{p'}{2} \, dx \\ & \approx \sup_{J \in \D} \frac{1}{|J|} \, \int_J \left( \sum_{\varepsilon \in \S} \sum_{I \in \D(J)} \frac{|V_I ' \V{f}_I ^\varepsilon|^2}{|I|} \, 1_I (x)  \right)^\frac{p'}{2} \, dx \\ & < \infty.   \end{align*}  Thus,  by Theorem $3.1$ in \cite{NTV}  we can assume that $\{V_I ^{-1}  \vec{f}_I ^\varepsilon\}_{\{I \in \D, \varepsilon \in \S\}} $ is a Carleson sequence. Now if $J$ is fixed, $I \subseteq J$,  and $B = V_J ^{-1} W^\frac{1}{p} 1_J $, then
\begin{align*}  |V_J ^{-1} \V{f}_I ^\varepsilon| = |V_J ^{-1} V_I V_I ^{-1} \V{f}_I ^\varepsilon| = |m_I (B W^{-\frac{1}{p}} ) V_I (V_I ^{-1} \V{f}_I ^\varepsilon)| \end{align*} so that \begin{align*} \sup_{J \in \D} & \frac{1}{|J|} \int_J \left(\sum_{\varepsilon \in \S} \sum_{I \in \D(J)} \frac{| V_J ^{-1}  \vec{f}_I ^\varepsilon|^2}{|I|} 1_I(x) \right)^\frac{p}{2} \, dx  \\ & \leq \sup_{J \in \D}  \frac{1}{|J|} \inrd \left(\sum_{\varepsilon \in \S} \sum_{I \in \D} \frac{|m_I (B W^{-\frac{1}{p}} ) V_I (V_I ^{-1} \V{f}_I ^\varepsilon)|^2}{|I|} 1_I(x) \right)^\frac{p}{2} \, dx \\ & \lesssim \sup_{J \in \D}  \frac{1}{|J|} \|V_J ^{-1} W^\frac{1}{p} 1_J \|_{L^p}^p  < \infty  \end{align*} by Lemma \ref{EmbeddingLem}.  \hfill $\square$

We will finish off the paper with a proof of Proposition \ref{CommProp}. \\

\noindent \textit{Proof of Proposition} \ref{CommProp} First we prove \eqref{VectorBMO1} for $q = p$ and $\vec{f}$ replaced by $B$.   By assumption, there exists $z_0 \neq 0 $  and $\delta > 0$ where $\frac{1}{K(x)}$ is smooth on $|x - z_0| < \sqrt{d} \delta$, and thus can be expressed as an absolutely convergent Fourier series \begin{equation*} \frac{1}{K(x)} = \sum a_k e^{i v_k \cdot x} \end{equation*} \noindent for $|x - z_0| < \sqrt{d} \delta$ (where the exact nature of the vectors $v_k$ is irrelevant.)  Set $z_1 = \delta^{-1} z_0$.  Thus, if $|x - z_1| < \sqrt{d}$, then we have by homogeneity

\begin{equation*} \frac{1}{K(x)} = \frac{\delta^{-d}}{K(\delta x)} = \delta^{-d} \sum a_n e^{i v_k \cdot (\delta x)} \end{equation*} \noindent Now for any cube $Q = Q(x_0, r)$ of side length $r$ and center $x_0$, let $y_0 = x_0 - rz_1$ and $Q' = Q(y_0, r)$ so that $x \in Q$ and $y \in Q'$ implies that \begin{equation*} \left|\frac{x - y}{r} - z_1\right|  \leq \left| \frac{x - x_0}{r} \right| + \left| \frac{y - y_0}{r} \right|\leq \sqrt{d}. \end{equation*}

Let \begin{equation*} S_Q (x) = \chi_Q (x) \frac{((B^*(x) - m_{Q'} B^* ) V_Q ^{-1} )^* } {\| (B^*(x) - m_{Q'} B^* ) V_Q ^{-1} \|} \end{equation*}  so that \begin{align}\frac{1}{r^d} &  \left\|\inrd  (B^*(x) - B^*(y)) V_Q ^{-1} \frac{r^d K(x - y)}{K(\frac{x-y}{r})}  S_Q (x)  \chi_{Q'} (y) \, dy  \right\| \label{CommEst1} \\ & = \chi_Q (x)  \frac{1}{r^d} \left\| \int_{Q'}   (B^*(x) - B^*(y)) V_Q ^{-1}  \frac{(  (B^*(x) - m_{Q'} B^* ) V_Q ^{-1}  )^* }{\|  (B^*(x) - m_{Q'} B^* ) V_Q ^{-1} \|}  \, dy \right\| \nonumber \\ & = \chi_Q (x) \left\|\frac{  (B^*(x) - m_{Q'} B^* ) V_Q ^{-1} (  (B^*(x) - m_{Q'} B^* ) V_Q ^{-1}  )^*}{\|   (B^*(x) - m_{Q'} B^* ) V_Q ^{-1}\|} \right\| \nonumber \\ & = \chi_Q (x) \| (B^*(x) - m_{Q'} B^* ) V_Q ^{-1}\| \nonumber \end{align}

However, \begin{align*} \eqref{CommEst1} &  \lesssim \sum_k |a_k| \left\|  \left(\inrd (B^*(x) - B^*(y)) K(x - y)  e^{- i \frac{\delta}{r} v_k \cdot y} V_Q ^{-1}  \chi_{Q'} (y) \, dy \right)S_Q (x)  e^{ i \frac{\delta}{r} v_k \cdot x} \right\|  \\ & \leq  \sum_k  |a_k| \left\|  \left(\inrd (B^*(x) - B^*(y)) K(x - y)  e^{- i \frac{\delta}{r} v_k \cdot y} V_Q ^{-1}  \chi_{Q'} (y) \, dy \right) \right\|
 \\ & \lesssim  \sum_k  \sum_{j = 1}^n |a_k| \left\|  ([T, B^*] (g_k \vec{e}_j)) (x)  \right\| \end{align*} where \begin{equation*} g_k (y) =   e^{ i \frac{\delta}{r} v_k \cdot y } V_Q ^{-1} \chi_{Q'} (y) \end{equation*}  and where the second inequality follows from the fact that $\| S_Q(x)  e^{ i \frac{\delta}{r} v_k \cdot x }\| \leq 1$ for a.e. $x \in \Rd$.

But as $|x_0 - y_0 | = r \delta^{-1} z_0$, we can pick some $C > 1$ only depending on $K$ where  $\tilde{Q} = Q(x_0, C r) $ satisfies $Q \cup Q' \subseteq \tilde{Q}$.  Combining this with the previous estimates, we have from the absolute summability of the $a_n's $ and the boundedness of $[T, B^*]$ from $L^p(W)$ to $L^p$ that \begin{align*} \left(\int_{Q} \|  V_Q ^{-1} (B(x) - m_{Q'} B ) \| ^p \, dx \right) ^{\frac{1}{p}} & = \left(\int_{Q} \|   (B^*(x) - m_{Q'} B^* ) V_Q ^{-1} \| ^p \, dx \right) ^{\frac{1}{p}} \\ & \leq \sum_k \sum_{j = 1}^n |a_k| \|  [T, B^*] (g_k \vec{e}_j)  \|_{L^p}  \\ & \lesssim   \sup_n  \sum_{j = 1}^n \| W^{\frac{1}{p}} g_k \vec{e_j} \|_{L^p } \\ & \leq \sum_{j = 1}^n \|\chi_{Q'}  W^{\frac{1}{p}} V_Q ^{-1} \vec{e}_j\|_{L^p} \\ & \lesssim   |Q|^\frac{1}{p} \|W\|_{\text{A}_p} ^\frac{1}{p} \end{align*} since the A${}_p$ condition gives us that \begin{align*}  \sum_{j = 1}^n \||Q|^{-\frac{1}{p}}\chi_{Q'}  W^{\frac{1}{p}} V_Q ^{-1} \V{e}_j \|_{L^p} & \lesssim   \sum_{j = 1}^n \||\tilde{Q}|^{-\frac{1}{p}} \chi_{\tilde{Q}} V_Q ^{-1} W^{\frac{1}{p}}  \V{e}_j \|_{L^p}  \lesssim  \sum_{j = 1}^n \| |\tilde{Q}|^{-\frac{1}{p}} \chi_{\tilde{Q}} V_{\tilde{Q}} '  W^{\frac{1}{p}}  \vec{e}_j \ \|_{L^p} \\ & \lesssim \|V_{\tilde{Q}}  V_{\tilde{Q}} ' \|  \lesssim \|W\|_{\text{A}_p} ^\frac{1}{p} \end{align*} A standard argument now shows \eqref{VectorBMO1} for $q = p$ and $\vec{f}$ replaced by $B$.

To prove \eqref{VectorBMO2} for $\vec{f}$ replaced by $B$, note that $[T, B^*]$ is bounded from $L^p(W)$ to $L^p$ if and only if $[T, B^*] W^{-\frac{1}{p}}$ is bounded on $L^p$ if and only if $W^{-\frac{1}{p}}[T^*, B]$ is bounded on $L^{p'}$.

Now let \begin{equation*} \W{S}_Q (x) = \chi_Q (x) \frac{(W^{-\frac{1}{p}} (x)  (B(x) - m_{Q'} B )  )^* }{\|W^{-\frac{1}{p}} (x)  (B(x) - m_{Q'} B )  \|}\end{equation*}  so that \begin{align}\frac{1}{r^d} &  \left\|\inrd  W^{-\frac{1}{p}} (x) (B(x) -B(y))  \frac{r^d \overline{K(x - y)}}{\overline{K(\frac{x-y}{r})}}  \W{S}_Q (x)  \chi_{Q'} (y) \, dy  \right\| \label{CommEst2} \\ & = \chi_Q (x)  \frac{1}{r^d} \left\| \int_{Q'}  W^{-\frac{1}{p}} (x) (B(x) -B(y))   \frac{(W^{-\frac{1}{p}} (x)  (B(x) - m_{Q'} B )  )^* }{\|W^{-\frac{1}{p}} (x)  (B(x) - m_{Q'} B )  \|}  \, dy \right\| \nonumber \\ & = \chi_Q (x) \left\|\frac{ W^{-\frac{1}{p}} (x)  (B(x) - m_{Q'} B ) (W^{-\frac{1}{p}} (x)  (B(x) - m_{Q'} B )   )^*}{\| W^{-\frac{1}{p}} (x)  (B(x) - m_{Q'} B ) \|} \right\| \nonumber \\ & = \chi_Q (x) \| W^{-\frac{1}{p}} (x)  (B(x) - m_{Q'} B ) \| \nonumber \end{align}

However, \begin{align*} \eqref{CommEst2} &  \leq \sum_k |a_k| \left\|  W^{-\frac{1}{p}} (x) \left(\inrd (B(x) - B(y)) \overline{K(x - y)}  e^{ i \frac{\delta}{r} v_k \cdot y}   \chi_{Q'} (y) \, dy \right)\W{S}_Q (x)  e^{ -i \frac{\delta}{r} v_k \cdot x} \right\|  \\ & \leq  \sum_k |a_k| \left\|  W^{-\frac{1}{p}} (x) \left(\inrd (B(x) - B(y)) \overline{K(x - y)}  e^{ i \frac{\delta}{r} v_k \cdot y}   \chi_{Q'} (y) \, dy \right) \right\|
 \\ & \lesssim  \sum_k  \sum_{j = 1}^n |a_k| \left\|  W^{-\frac{1}{p}} (x) ([T^*, B] (\W{g}_k \vec{e}_j)) (x)  \right\| \end{align*} where \begin{equation*} \W{g}_k (y) =   e^{- i \frac{\delta}{r} v_k \cdot y }  \chi_{Q'} (y) \end{equation*}  and where the second inequality follows from the fact that $\| \W{S}_Q(x)  e^{ i \frac{\delta}{r} v_k \cdot x }\| \leq 1$ for a.e. $x \in \Rd$.

Then again we have \begin{align*} \left(\int_{Q} \|  W^{-\frac{1}{p}} (x) (B(x) - m_{Q'} B ) \| ^{p'} \, dx \right) ^{\frac{1}{p'}} &  \lesssim \sum_k \sum_{j = 1}^n |a_k| \| W^{-\frac{1}{p}} [T^*, B] (\W{g}_k \vec{e}_j)  \|_{L^{p'}}  \\ & \lesssim   \sup_k  \sum_{j = 1}^n \| W^{-\frac{1}{p}} (x) [T^*, B] (\W{g}_k \vec{e}_j)  \|_{L^{p'}} \\ & \leq \sum_{j = 1}^n \|\chi_{Q'}  \vec{e}_j\|_{L^{p'}} \\ & \lesssim   |Q|^\frac{1}{p'}  \end{align*}

   Finally, we can use a simple argument from \cite{IM} to get \begin{align*}   \left(\frac{1}{|Q|} \int_{Q} \|  W^{-\frac{1}{p}} (x) (B(x) - m_{Q} B ) \| ^{p'} \, dx \right) ^{\frac{1}{p'}} & \leq  \left(\frac{1}{|Q|} \int_{Q} \|  W^{-\frac{1}{p}} (x) (B(x) - m_{Q'} B ) \| ^{p'} \, dx \right) ^{\frac{1}{p'}}  \\ & +  \left(\int_{Q} \|  W^{-\frac{1}{p}} (x) (m_{Q'} B - m_{Q} B ) \| ^{p'} \, dx \right) ^{\frac{1}{p'}}\end{align*}  where \begin{align*} & \left(\frac{1}{|Q|}\int_{Q} \|  W^{-\frac{1}{p}} (x) (m_{Q'} B - m_{Q} B ) \| ^{p'} \, dx \right) ^{\frac{1}{p'}} \\ & =  \left(\frac{1}{|Q|} \int_{Q} \left\| \frac{1}{|Q|} \int_{Q}  W^{-\frac{1}{p}} (x)  ( B(y) - m_{Q'} B )  \, dy \right\| ^{p'} \, dx \right) ^\frac{1}{p'} \\ & \leq \left(\frac{1}{|Q|} \int_{Q} \left( \frac{1}{|Q|} \int_{Q} \| W^{-\frac{1}{p}} (x) W^{\frac{1}{p}} (y)\| \, \|  W^{-\frac{1}{p}} (y) ( B(y) - m_{Q'} B ) \|  \, dy \right) ^{p'} \, dx \right) ^\frac{1}{p'}     \\ & \leq \left(\frac{1}{|Q|} \int_{Q} \left(\frac{1}{|Q|} \int_{Q}  \| W^{-\frac{1}{p}} (x) W^{\frac{1}{p}} (y) \|^{p} \, dy \right)^\frac{p'}{p} \, dx \right)^\frac{1}{p'}  \\ & \times  \left(\frac{1}{|Q|} \int_Q  \| W^{-\frac{1}{p}}(y)(B(y) - m_{Q'} B ) \|^{p'} \, dy \right) ^\frac{1}{p'} \\ & \leq \|W\|_{\text{A}_p}  \left(\frac{1}{|Q|} \int_Q  \| W^{-\frac{1}{p}}(y)(B(y) - m_{Q'} B )  \|^{p'} \, dy \right) ^\frac{1}{p'} \end{align*}  which completes the proof. \hfill $\square$

     \end{document}